\documentclass[11pt,a4paper,twoside]{amsart}
\usepackage[french,english]{babel}
\usepackage[T1]{fontenc}
\ifx\pdfpageheight\undefined\PassOptionsToPackage{dvips}{graphicx}%
\else%
\usepackage{amsfonts}
\usepackage[latin1]{inputenc}

\title{Th\'eor\`emes de Borel avec contraintes}

\author{D. Cerveau, D. Garba Belko  }
\date{}

\DeclareSymbolFont{lasy}{U}{lasy}{m}{n}
\SetSymbolFont{lasy}{bold}{U}{lasy}{b}{n}
\let\Box\undefined
\DeclareMathSymbol\Box{\mathord}{lasy}{"32}

\newtheorem{theorem}{Th\'eor\`eme}
\newtheorem{proposition}[theorem]{Proposition}
\newtheorem{lemma}[theorem]{Lemme}
\newtheorem{corollary}[theorem]{Corollaire}
\newtheorem{remark}[theorem]{Remarque}
\newtheorem{example}[theorem]{Exemple}
\newtheorem{examples}[theorem]{Exemples}
\newtheorem{definition}[theorem]{D\'efinition}

\newtheorem*{conjecture}{Conjecture}

\newcommand {\junk}[1]{}

\makeatletter
\def\maketitle{\par
\begingroup
\def\@makefnmark{\hbox
to 0pt{$^{\@thefnmark}$\hss}} \if@twocolumn
\twocolumn[\@maketitle] \else \newpage \global\@topnum\z@
\@maketitle \fi\thispagestyle{plain}\@thanks
\endgroup
\setcounter{footnote}{0}
\let\maketitle\relax
\let\@maketitle\relax
\gdef\@thanks{}\gdef\@author{}\gdef\@title{}\let\thanks\relax}
\makeatother

\def\.@{\char'76}

\def\.@{\char'76}

\begin{document}

\selectlanguage{french}
\begin{abstract} Un th\'eor\`eme classique de Borel affirme que chaque s\'erie formelle \`a coefficients r\'eels est le jet Taylorien d'un germe de fonction $\mathcal{C}^{\infty}$. Nous \'etudions ce type de probl\`eme en particulier pour des alg\`ebres de Lie de champs de vecteurs ou des groupes de diff\'eomorphismes.

\selectlanguage{english}
A classical theorem due to Borel asserts that any formal serie with real coefficients is the Taylor expansion of a germ of $\mathcal{C}^{\infty}- {\rm function}$. We study such a problem in the context of Lie algebras of vector fields or of groups of diffeomorphisms.
\end{abstract}
\maketitle


\selectlanguage{french}

\section*{Introduction} \label{sec Introduction}
\addcontentsline{toc}{section}{Introduction}

Soient $\mathcal{E}_{n}$ l'anneau des germes de fonctions   $\mathcal{C}^{\infty}$ \`a l'origine $\underline{0}$ de $\mathbb{R}^{n}$
et $\hat{\mathcal{E}}_{n}$ celui des s\'eries formelles en  $n$ ind\'etermin\'ees et \`a coefficients r\'eels. L'application:

$$T_{\underline{0}}: \mathcal{E}_{n} \rightarrow \hat{\mathcal{E}}_{n}$$
qui \`a une fonction $f$ associe son jet de Taylor infini est un morphisme d'anneau. Le th\'eor\`eme de r\'ealisation de Borel affirme que $T_{\underline{0}}$ est surjectif:
 si $\hat{f} \in \hat{\mathcal{E}}_{n}$ est une s\'erie formelle, il existe un germe de fonction $f \in \mathcal{E}_{n}$ tel que  $T_{\underline{0}} f = \hat{f}$.
 Un tel $f$ sera dit une r\'ealisation $\mathcal{C}^{\infty}$ de $\hat{f}$.   Deux r\'ealisations de $\hat{f}$ diff\`erent par une fonction plate en $\underline{0}$.
 Si l'on note $\mathcal{M}_{n} \subset \mathcal{E}_{n}$ l'id\'eal maximal constitu\'e des germes s'annulant au point $\underline{0}$, alors le noyau de
 $T_{\underline{0}}$ est exactement l'id\'eal des fonctions plates $Ker T_{\underline{0}} = \mathcal{M}_{n}^{\infty} = \bigcap_{k} \mathcal{M}_{n}^{k}$.
 L'application $T_{\underline{0}}$ s'\'etend naturellement aux $\mathcal{E}_{n}$-modules "classiques". Si $\Omega_{n}^{k}$ d\'esigne le $\mathcal{E}_{n}$-module des
 germes de $k$-formes diff\'erentielles et $\hat{\Omega}_{n}^{k}$ le $\hat{\mathcal{E}}_{n}$-module des $k$-formes formelles on d\'esignera encore
 $T_{\underline{0}}: \Omega_{n}^{k} \rightarrow \hat{\Omega}_{n}^{k}$ le morphisme qui \`a une $k$-forme $\mathcal{C}^{\infty}$ associe son jet Taylorien infini en
 l'origine. De m\^eme si $\mathcal{X}_{n}$ (resp. $\hat{\mathcal{X}}_{n}$) d\'esigne l'alg\`ebre de Lie des germes de champs de vecteurs $\mathcal{C}^{\infty}$
 (resp. formels) et $\mathrm{Diff}(\mathbb{R}^{n}_{0})$ (resp. $\widehat{\mathrm{Diff}}(\mathbb{R}^{n}_{0})$) celui des germes de diff\'eomorphisme $\mathcal{C}^{\infty}$
 (resp. formels), en $\underline{0} \in \mathbb{R}^{n}$, on dispose encore de morphismes de jets Tayloriens
 $T_{\underline{0}}: \mathcal{X}_{n} \rightarrow \hat{\mathcal{X}}_{n}$ et $T_{\underline{0}}: \mathrm{Diff}(\mathbb{R}^{n}_{0}) \rightarrow \widehat{\mathrm{Diff}}(\mathbb{R}^{n}_{0})$.
 Dans le premier cas c'est un morphisme d'alg\`ebre de Lie et dans le second un morphisme de groupe. Le th\'eor\`eme de Borel s'implante brutalement sur ces espaces, i.e.
 tout \'el\'ement de $\hat{\Omega}_{n}^{k}$, $\hat{\mathcal{X}}_{n}$ ou $\widehat{\mathrm{Diff}}(\mathbb{R}^{n}_{0})$ poss\`ede une r\'ealisation $\mathcal{C}^{\infty}$ dans les
 espaces correspondants. Mais ces espaces poss\`edent des structures suppl\'ementaires, produit ext\'erieur et op\'erateur $d$ pour les $\Omega_{n}^{k}$, crochet de Lie
 pour les champs de vecteurs et composition pour les diff\'eomorphismes. Se posent alors les probl\`emes de r\'ealisation de type Borel tenant compte de ces structures. Ce que nous appelons le th\'eor\`eme de Borel avec contraintes. En voici quelques exemples:
\begin{enumerate}
\item Etant donn\'e, $\hat{G} \subset \widehat{\mathrm{Diff}}(\mathbb{R}^{n}_{0})$, un sous groupe de type fini de diff\'eomor-phismes formels, existe-t-il une r\'ealisation $G \subset \mathrm{Diff}(\mathbb{R}^{n}_{0})$ telle que la restriction $T_{\underline{0}}: G \rightarrow \hat{G}$ soit un isomorphisme de groupe?
\item Soit $\hat{\mathcal{G}} \subset \hat{\mathcal{X}}_{n}$ une sous alg\`ebre de Lie de champs de vecteurs formels de dimension finie. Existe-t-il une r\'ealisation $\mathcal{G} \subset \mathcal{X}_{n}$ telle que la restriction $T_{\underline{0}}: \mathcal{G} \rightarrow \hat{\mathcal{G}}$ soit un isomorphisme d'alg\`ebre de Lie?
\item Soit $\hat{\omega}$ une 1-forme int\'egrable formelle non triviale. Peut on cette fois trouver une r\'ealisation $\omega \in \Omega^{1}_{n}$ de $\hat{\omega}$ qui soit int\'egrable, i.e. $d \omega \wedge \omega = 0$? Un probl\`eme analogue se pose pour les syst\`emes de Pfaff.
\end{enumerate}

Les probl\`emes de type Borel  ont int\'eress\'e de nombreux math\'ematiciens. C'est ainsi que dans le cadre de l'\'etude des alg\`ebres quasi-analytiques J. C. Tougeron \cite{Tougeron3}
montre que le morphisme  $T_{\underline{0}}: \mathcal{E}_{1} \rightarrow \hat{\mathcal{E}}_{1}$ poss\`ede des sections. Toutefois ces sections ne respectent pas la
composition.
En dimension 2, o\`u la condition d'int\'egrabilit\'e est triviale, R. Roussarie \cite{Roussarie} a donn\'e plusieurs r\'esultats de type Borel.

Sans r\'esoudre, en toute g\'en\'eralit\'e, les probl\`emes \'enum\'er\'es ci-dessus nous apportons des r\'eponses positives dans quelques cas particuliers. Ces r\'eponses sont parfois
des adaptations de r\'esultats relativement classiques (d\'etermination finie par exemple) ou n\'ecessitant des techniques sp\'ecifiques.

 \section{Alg\`ebres de Lie de champs de vecteurs}

 \subsection{Alg\`ebres semi-simples, alg\`ebres de rang ponctuel 1, alg\`e-bres saturables}

 Soient $\hat{\mathcal{L}}$ une sous alg\`ebre de Lie de $\hat{\mathcal{X}}_{n}$ et $\hat{\mathcal{L}} (\underline{0}) = \{ \hat{X} (\underline{0})   / \hat{X} \in \hat{\mathcal{L}} \}$
 l'\'evaluation de $\hat{\mathcal{L}}$ en $\underline{0}$. Nous nous int\'eressons au cas purement singulier o\`u $\hat{\mathcal{L}} (\underline{0}) = \{ \underline{0} \}$.
 Sous cette hypoth\`ese l'ensemble ${\mathcal{L}}^{1} = \{J^{1} X / X \in \hat{\mathcal{L}} \}$ "des parties lin\'eaires" des \'el\'ements  de $\hat{\mathcal{L}}$ est une
 sous alg\`ebre de Lie de l'alg\`ebre de Lie  $\mathcal{X}^{1}_{n}$ des champs de vecteurs lin\'eaires de $\mathbb{R}^{n}$. Notons que $\mathcal{X}^{1}_{n}$ est
 isomorphe \`a l'espace vectoriel des endomorphismes, $\mathrm{End} \mathbb{R}^{n}$, de $\mathbb{R}^{n}$. \\

 Supposons que $\hat{\mathcal{L}}$  soit semi-simple, i.e. $\hat{\mathcal{L}}$  n'admet pas d'id\'eal r\'esoluble non nul. D'apr\`es un r\'esultat de R. Hermann \cite{Hermann} $J^{1} : \hat{\mathcal{L}} \rightarrow \mathcal{L}^{1}$ est injectif et  $\hat{\mathcal{L}}$ est formellement lin\'earisable. Ceci signifie qu'il existe $\hat{\Phi} \in \widehat{\mathrm{Diff}}(\mathbb{R}^{n}_{0})$ qui conjugue $\hat{\mathcal{L}}$ \`a $\mathcal{L}^{1}$:
 $$\hat{\mathcal{L}} = \hat{\Phi}_{\ast} \hat{\mathcal{L}}^{1} = \{ \hat{\Phi}_{\ast} (J^{1} X ) / X \in \hat{\mathcal{L}} \}.$$

 Soit $\Phi$ une r\'ealisation $\mathcal{C}^{\infty}$ de $\hat{\Phi}$; l'alg\`ebre de Lie $\mathcal{L} = \Phi_{\ast} \mathcal{L}^{1}$ est une r\'ealisation de
 $\hat{\mathcal{L}}$ et par construction $T_{\underline{0}}: \mathcal{L} \rightarrow \hat{\mathcal{L}}$ est un isomorphisme. D'o\`u le:

 \begin{theorem} Soit $\hat{\mathcal{L}} \subset \mathcal{M}_{n} \hat{\mathcal{X}}_{n}$ une alg\`ebre de Lie semi-simple de champ de vecteurs formels.  Alors
 $\hat{\mathcal{L}}$ poss\`ede une r\'ealisation $\mathcal{C}^{\infty}$ not\'ee $\mathcal{L}$ telle que \\ $T_{\underline{0}}: \mathcal{L} \rightarrow \hat{\mathcal{L}}$ soit
 un isomorphisme.
\end{theorem}

Comme ci-dessus toutes les  alg\`ebres de Lie lin\'earisables de champs formels poss\`edent une r\'ealisation $\mathcal{C}^{\infty}$; de m\^eme celles qui sont conjugu\'ees \`a
une alg\`ebre de champs polynomiaux. C'est le cas en petite dimension $n$ d'espace. En dimension $n = 1$, la classification formelle des sous alg\`ebres
$\hat{\mathcal{L}}$ de dimension finie de $\hat{\mathcal{X}}_{1}$ fait partie du folklore. Elle est probablement connue  de S. Lie, F. Klein et E. Cartan.
\begin{enumerate}
\item $n = 1$ et $\dim \hat{\mathcal{L}} = 1$; $\hat{\mathcal{L}}$ est formellement conjugu\'ee \`a l'alg\`ebre engendr\'ee par l'un des champs $\frac{\partial}{\partial x}$ et $X_{p , \lambda} = \frac{x^{p + 1}}{1 - \lambda x^{p}} \frac{\partial}{\partial x}$ avec $p \in \mathbb{N}$ et $\lambda \in \mathbb{R}$.
\item $n = 1$ et $\dim \hat{\mathcal{L}} = 2$; $\hat{\mathcal{L}}$ est formellement conjugu\'ee \`a l'une des alg\`ebres $\langle \frac{\partial}{\partial x} , x \frac{\partial}{\partial x} \rangle$ et $\langle  x \frac{\partial}{\partial x} , x^{p} \frac{\partial}{\partial x} \rangle$ avec $p \in \mathbb{N} \setminus \{ 0 \}$. Toutes ces alg\`ebres sont isomorphes \`a l'alg\`ebre du groupe des transformations affines de la droite.
\item $n = 1$ et $\dim \hat{\mathcal{L}} = 3$; $\hat{\mathcal{L}}$ est formellement conjugu\'ee \`a l'alg\`ebre \\ $\langle \frac{\partial}{\partial x} , x \frac{\partial}{\partial x} , x^{2} \frac{\partial}{\partial x} \rangle$ qui est l'alg\`ebre du groupe des transformations homographiques $\mathbb{P }G  L ( 2 , \mathbb{R} )$.
\end{enumerate}

Toutes ces alg\`ebres $\hat{\mathcal{L}}$ sont formellement conjugu\'ees \`a des alg\`ebres de champs analytiques $\mathcal{L}^{an}$. Le th\'eor\`eme de Borel usuel, appliqu\'e \`a
une conjuguante $\hat{\Phi}$ ($\hat{\mathcal{L}} = \hat{\Phi}_{\star} \mathcal{L}^{an}$) produit une r\'ealisation $\mathcal{L} = \Phi_{\star} \mathcal{L}^{an}$ des $\hat{\mathcal{L}}$
consid\'er\'ees. En dimension plus grande la classification des alg\`ebres de champs formels n'est pas connue. En dimension deux on peut faire la liste de celles qui sont
formellement lin\'earisables (pour lesquelles on aura donc des \'enonc\'es de type Borel). Ceci met en jeu des conditions de non r\'esonance, \`a la Poincar\'e, portant sur leur
radical r\'esoluble. En un certain sens cette liste de type "zoologique" ne met \`a jour ni des techniques nouvelles, ni des r\'esultats nouveaux. Pour illustrer ce qui
pr\'ec\`ede nous allons traiter quelques cas sp\'eciaux en dimension 2 (d'espace); en particulier  celui des alg\`ebres commutatives. Nous avons pour cela besoin de la
notion de rang ponctuel g\'en\'erique ($\mathcal{r} ( \hat{\mathcal{L}} )$) que nous d\'efinissons pour n'importe quelle alg\`ebre $\hat{\mathcal{L}}$ de $\hat{\mathcal{X}}_{n}$.

\begin{definition} Soit $\hat{\mathcal{L}}  \subset \hat{\mathcal{X}}_{n}$ une sous alg\`ebre non nulle; le rang ponctuel g\'en\'erique
(ou plus simplement le rang) $\mathcal{r} (\hat{\mathcal{L}})$ est le nombre maximal $k$ d'\'el\'ements, $\hat{X}_{1} ,  \ldots  , \hat{X}_{k}$ de $ \hat{\mathcal{L}}$,
qui sont  $\hat{\mathcal{E}}_{n}$-ind\'ependants. Si $\hat{X}_{j} = \sum \hat{a}_{i , j} \frac{\partial}{\partial x_{i}}$, $\hat{a}_{i , j} \in \hat{\mathcal{E}}_{n}$, $j = 1 ,  \ldots  , \mathcal{r} ( \hat{\mathcal{L}} )$, la matrice $(\hat{a}_{i , j})$ poss\`ede un mineur $\mathcal{r} (\hat{\mathcal{L}}) \times \mathcal{r} (\hat{\mathcal{L}})$ de d\'eterminant non nul.
\end{definition}

Le rang ponctuel g\'en\'erique est trivialement major\'e par la dimension ambiante. Par exemple l'alg\`ebre
$\hat{\mathcal{L}} = \{ \hat{f}(x_{2}) \frac{\partial}{\partial x_{1}} / \hat{f} \in \hat{\mathcal{E}}_{1} \}$ est une sous alg\`ebre de Lie de $\hat{\mathcal{E}}_{2} $
de dimension infinie et de rang $1$. Notons que $\hat{\mathcal{L}}$ est commutative et que le th\'eor\`eme de Borel, appliqu\'e aux $\hat{f}$, donne une r\'ealisation
$\mathcal{C}^{\infty}$ de $\hat{\mathcal{L}}$. L'alg\`ebre $\mathcal{L} = \{ f(x_{2}) \frac{\partial}{\partial x_{1}} / f \in \mathcal{E}_{1} \}$  est commutative et
se projette sur $\hat{\mathcal{L}}$ ($T_{\underline{0}} \mathcal{L} = \hat{\mathcal{L}}$), mais $T_{\underline{0}} : \mathcal{L} \rightarrow \hat{\mathcal{L}}$ n'est
pas un isomorphisme puisque si $P \in \mathcal{E}_{1}$ est une fonction plate, alors $T_{\underline{0}} (P(x_{2})\frac{\partial}{\partial x_{1}} = 0$. En fait
choisissons une $\mathbb{R}$-base $\{ \hat{a}_{i} , i \in I \}$ du $\mathbb{R}$-espace vectoriel $\hat{\mathcal{E}}_{1}$ et soient $a_{i}\in \mathcal{E}_{1}$, $i \in I$,
des r\'ealisations de Borel des $\hat{a}_{i}$. Alors le $\mathbb{R}$-espace vectoriel $\mathcal{E}^{'}$ engendr\'e par les $a_{i}$ produit une r\'ealisation
$\mathcal{L}^{'} = \{ f(x_{2}) \frac{\partial}{\partial x_{1}} / i \in I \}$ pour laquelle $T_{\underline{0}} : \mathcal{L}^{'} \rightarrow \hat{\mathcal{L}}$ est
un isomorphisme.

\subsection{Description des sous alg\`ebres $\hat{\mathcal{L}} \subset \hat{\mathcal{X}}_{n}$ de dimension finie et de rang 1}

Soient $\hat{X}$ un champ formel et $E \subset \hat{\mathcal{E}}_{n}$ un sous espace vectoriel ayant la propri\'et\'e suivante:

$$\forall  (\hat{f} , \hat{g}) \in E \times E , \hat{f} \hat{X} (\hat{g}) - \hat{g} \hat{X} (\hat{f}) \in E  \eqno{(\star)}$$

o\`u $\hat{f} \mapsto \hat{X} ( \hat{f} )$ est la d\'erivation associ\'ee \`a $\hat{X}$. Alors l'alg\`ebre
 $\hat{\mathcal{L}} = E . \hat{X} = \{ \hat{f} \hat{X} / \hat{f}
\in E \}$ est de rang $1$ et de dimension celle de $E$. En fait toute sous alg\`ebre de Lie  $\hat{\mathcal{L}}$ de rang $1$ s'obtient ainsi. En effet si $\hat{Y} = \sum \hat{a}_{i } \frac{\partial}{\partial x_{i}}$, $\hat{a}_{i } \in \hat{\mathcal{E}}_{n}$ est un \'el\'ement non nul, alors le champ $\hat{X} = \frac{\hat{Y}}{\mathrm{pgcd} ( \hat{a}_{1} , \ldots , \hat{a}_{n} )}$ convient. A noter que le champ $\hat{X}$ n'appartient peut \^etre pas \`a $\hat{\mathcal{L}}$.

L'alg\`ebre $\hat{\mathcal{L}}$ de rang $1$ sera dite \textit{saturable} s'il existe $\hat{X} = \sum \hat{a}_{i } \frac{\partial}{\partial x_{i}} \in \hat{\mathcal{L}}$
satisfaisant $\mathrm{pgcd} ( \hat{a}_{1} , \ldots , \hat{a}_{n} ) = 1$. Dans ce cas $\hat{\mathcal{L}} = \{ \hat{f} \hat{X} / \hat{f} \in E \}$ satisfait la condition
$(\ast)$. L'alg\`ebre de Lie $\langle \frac{\partial}{\partial x} , x \frac{\partial}{\partial x} , x^{2} \frac{\partial}{\partial x} \rangle$ est saturable de rang
$1$; ici $\hat{X} = \frac{\partial}{\partial x}$. Par contre l'alg\`ebre $\langle  x \frac{\partial}{\partial x} , x^{2} \frac{\partial}{\partial x} \rangle$ n'est pas
saturable. On obtient d'autres exemples d'alg\`ebres saturables; pour cela notons $R_{n} = \sum x_{i} \frac{\partial}{\partial x_{i}} \in \hat{\mathcal{X}}_{n}$
le champ "radial" et $E_{n}^{d}$ l'espace vectoriel des polyn\^omes homog\`enes de degr\'e $d$ en les variables
$x_{1} , \ldots , x_{n}$. Les alg\`ebres de Lie $\mathcal{R}_{n}^{d} :=  E_{n}^{d} R_{n} = \{ f . R_{n} / f \in E_{n}^{d} \}$ sont commutatives de rang ponctuel
g\'en\'erique $1$. Parmi ces  alg\`ebres seules les $\mathcal{R}_{n}^{0}$, $n \geq 2$, sont saturables. On peut fabriquer d'autres sous alg\`ebres de $\hat{\mathcal{X}}_{n}$
au moyen de $\mathcal{R}_{n}^{d}$. Par exemple les alg\`ebres $\overline{\mathcal{R}}_{n}^{d}$, engendr\'ees par le champ radial et $\mathcal{R}_{n}^{d}$, sont
r\'esolubles de rang $1$. Les sous alg\`ebres de $\mathrm{GL} (n, \mathbb{R}) \subset \mathrm{End} \mathbb{R}^{n}$, en particulier les $\mathfrak{sl} (n, \mathbb{R})$, peuvent \^etre vues
comme des alg\`ebres de champs de vecteurs lin\'eaires sur $\mathbb{R}^{n}$.

Notons $\mathfrak{sl} \mathbb{R}^{d}_{n}$ le $\mathbb{R}$-espace vectoriel des champs de vecteurs engendr\'e par $\mathfrak{sl} (n, \mathbb{R})$ et $\mathbb{R}^{d}_{n}$; c'est en fait
une sous alg\`ebre de Lie dont la d\'ecomposition de Levi-Malcev est: $\mathfrak{sl} \mathbb{R}^{d}_{n}  = \mathbb{R}^{d}_{n} \oplus \mathfrak{sl} (n, \mathbb{R})$. Ces alg\`ebres sont de
rang maximal $n$. De m\^eme les alg\`ebres $\mathfrak{sl}  \overline{\mathbb{R}}^{d}_{n}$ engendr\'ees par $\mathcal{R}_{n}$ et $\mathfrak{sl}  \mathbb{R}^{d}_{n} $ ont leur d\'ecomposition de
Levi-Malcev du type suivant: $\mathfrak{sl}  \overline{\mathbb{R}}^{d}_{n}  = \overline{\mathbb{R}}^{d}_{n} \oplus \mathfrak{sl} (n, \mathbb{R})$. La d\'efinition suivante est naturelle:

\begin{definition} Soit $\mathcal{L}^{d}$ une sous alg\`ebre de Lie de $\hat{\mathcal{X}}_{n}$ dont les \'el\'ements sont des champs de vecteurs polyn\^omiaux de degr\'e
inf\'erieur ou \'egal \`a $d$. On dit que $\mathcal{L}^{d}$ est $d$-d\'eterminante si pour toute sous alg\`ebre $\hat{\mathcal{L}} \subset \hat{\mathcal{E}}_{n}$ ayant
$\mathcal{L}^{d}$ pour $d$-jet et telle que l'application jet d'ordre $d$ ($J^{d} : \hat{\mathcal{L}} \rightarrow \mathcal{L}^{d}$) soit un isomorphisme d'alg\`ebre de
Lie, alors $\hat{\mathcal{L}}$ et $\mathcal{L}^{d}$ sont conjugu\'ees: il existe $\hat{\Phi}$ appartenant \`a $\widehat{\mathrm{Diff}}(\mathbb{R}^{n}_{0})$ tel que
$\hat{\Phi}_{\ast} \hat{\mathcal{L}} = \mathcal{L}^{d}$.
\end{definition}

\begin{theorem} Les alg\`ebres $\overline{\mathbb{R}}^{d}_{n}$ et $\mathfrak{sl}  \overline{\mathbb{R}}^{d}_{n}$ sont $d + 1$ d\'eterminantes.
\end{theorem}

\begin{proof}
Soit $\hat{\mathcal{L}} \subset \hat{\mathcal{X}}_{n}$ une sous alg\`ebre telle que $J^{d + 1} : \hat{\mathcal{L}} \rightarrow \mathcal{L}^{d + 1}$ soit un
isomorphisme avec $\mathcal{L}^{d + 1}$ \'egal \`a $\overline{\mathbb{R}}^{d}_{n}$ ou $\mathfrak{sl}  \overline{\mathbb{R}}^{d}_{n}$. Il existe donc un \'el\'ement $\hat{R}$ de
$\hat{\mathcal{L}}$ dont le jet d'ordre $d + 1$ est pr\'ecis\'ement $R_{n}$. Le th\'eor\`eme de lin\'earisation de Poincar\'e produit un diff\'eomorphisme formel $\hat{\Phi}$
v\'erifiant $J^{ d + 1} \hat{\Phi} = \mathrm{Id}_{\mathbb{R}^{n}}$ et $\hat{\Phi}_{\ast} \check{R} = R_{n}$, o\`u $\mathrm{Id}_{\mathbb{R}^{n}}$ d\'esigne l'identit\'e de $\mathbb{R}^{n}$.
Soit $\hat{X}$ un \'el\'ement de $\hat{\mathcal{L}}$ tel que $J^{d + 1} \hat{X} = X_{d}$ appartient \`a $\mathcal{R}_{n}^{d}$  ou \`a
$\mathfrak{sl}  \overline{\mathbb{R}}^{d}_{n}$. Comme $[ R , X_{d} ] = d X_{d}$ on a  $[ \hat{R }, \hat{X} ] = d \hat{X}$ puisque $J^{d + 1}$ est un isomorphisme.
On en d\'eduit que:

$$[\hat{\Phi}_{\ast} \hat{R} , \hat{\Phi}_{\ast} \hat{X} ] = d \hat{\Phi}_{\ast} \hat{X} . $$

Un calcul \'el\'ementaire montre que le champ formel $\hat{\Phi}_{\ast}\hat{X}$ est homog\`ene de degr\'e $d + 1$; comme $J^{d + 1} \hat{\Phi} = \mathrm{Id}_{\mathbb{R}^{n}}$,
$\hat{\Phi}_{\ast} \hat{X} = X_{d}$. Ainsi $\hat{\Phi}_{\ast} \hat{\mathcal{L}} = \mathcal{L}^{d + 1}$. 
\end{proof}

On d\'eduit du th\'eor\`eme pr\'ec\'edent l'existence de $\hat{\Phi}$ tel que $\hat{\Phi}_{\ast} \hat{\mathcal{L}} = \mathcal{L}^{d + 1}$. Soit $\Phi$ une
r\'ealisation de Borel de $\hat{\Phi}$, et $\mathcal{L} = (\Phi^{-1})_{\ast}  \mathcal{L}^{d + 1}$. Alors $\mathcal{L}$ est une r\'ealisation de $\hat{\mathcal{L}}$
telle que $T_{\underline{0}} : \mathcal{L} \rightarrow \hat{\mathcal{L}}$ soit un isomorphisme; d'o\`u le:

\begin{corollary} Soit $\hat{\mathcal{L}} \subset \hat{\mathcal{X}}_{n}$ une sous alg\`ebre de Lie. On suppose  que $J^{d + 1} \hat{L} = \mathcal{L}^{d + 1}$ est
une sous alg\`ebre de Lie et que $J^{d + 1} : \hat{\mathcal{L}} \rightarrow \mathcal{L}^{d + 1}$ est un isomorphisme. Si $\mathcal{L}^{d + 1}$ est \'egal \`a
$\overline{\mathbb{R}}^{d}_{n}$ ou $\mathfrak{sl}  \overline{\mathbb{R}}^{d}_{n}$, alors $\hat{\mathcal{L}}$ poss\`ede une r\'ealisation $\mathcal{C}^{\infty}$ $\mathcal{L}$ telle
que $T_{\underline{0}} : \mathcal{L} \rightarrow \hat{\mathcal{L}}$ soit un isomorphisme.
\end{corollary}

Consid\'erons \`a pr\'esent une alg\`ebre de Lie  saturable,  de rang $1$ et de dimension finie:
$\hat{\mathcal{L}} = E. \hat{X} = \langle \hat{X} , \hat{f}_{1}\hat{X} \ldots , \hat{f}_{p}\hat{X} \rangle$, o\`u $\hat{f}_{k} \in \hat{\mathcal{E}}_{n}$.
Ici $\dim \hat{\mathcal{L}} = \dim E = p + 1$ que l'on suppose sup\'erieure ou \'egale \`a $2$. Le fait que $\hat{\mathcal{L}}$ soit une alg\`ebre de Lie implique que le sous
espace vectoriel $E = \langle 1 , \hat{f}_{1} \ldots , \hat{f}_{p} \rangle$ est invariant sous l'action de la d\'erivation $\hat{X}$; ce qui conduit au syst\`eme
diff\'erentiel:

\begin{equation} \mathcal{D} (\hat{X})
\left\{
  \begin{array}{ccc}
   \hat{X} (\hat{f}_{1}) & = & \sum_{j = 0}^{p} \lambda_{1}^{j} \hat{f}_{j}\\
   . &        . & \\
  \hat{X} (\hat{f}_{p}) & = & \sum_{j = 0}^{p} \lambda_{p}^{j} \hat{f}_{j}
  \end{array}
\right.
\end{equation}
o\`u l'on a pos\'e $\hat{f}_{0} = 1$. On distingue plusieurs cas suivant la nature du premier jet non nul du champ $\hat{X}$.

\subsubsection{Alg\`ebres saturables de rang 1 non singuli\`eres}

Dans ce cas le champ formel $\hat{X}$ est non singulier. Il existe, en particulier, $\hat{\Phi}$ appartenant \`a $\widehat{\mathrm{Diff}}(\mathbb{R}^{n}_{0})$ tel que
$\hat{\Phi}_{\ast} \hat{X} = \frac{\partial}{\partial x_{1}}$ et $\hat{\Phi}_{\ast} \hat{\mathcal{L}} = E \circ \hat{\Phi}^{- 1} .
\frac{\partial}{\partial x_{1}} = \langle \frac{\partial}{\partial x_{1}} , \hat{g}_{1} \frac{\partial}{\partial x_{1}} \ldots , \hat{g}_{p} \frac{\partial}{\partial x_{1}} \rangle$
avec $\hat{g}_{i} \circ \hat{\Phi} = \hat{f}_{i}$. Le syst\`eme diff\'erentiel $\mathcal{D} (\hat{\Phi}_{\ast} \hat{X}) = \mathcal{D}(\frac{\partial}{\partial x_{1}})$
devient maintenant un syst\`eme d'\'equations diff\'erentielles ordinaires:

\begin{equation}
\mathcal{D}(\frac{\partial}{\partial x_{1}})
\left\{
  \begin{array}{ccc}
   \frac{\partial \hat{g}_{1}}{\partial x_{1}} & = & \sum_{j = 0}^{p} \lambda_{1}^{j} \hat{g}_{j}\\
   . &        . & \\
  \frac{\partial \hat{g}_{p}}{\partial x_{1}} & = & \sum_{j = 0}^{p} \lambda_{p}^{j} \hat{g}_{j}
  \end{array}
\right.
\end{equation}

Consid\'erons le comme un syst\`eme diff\'erentiel en une variable $x_{1}$. Il poss\`ede un syst\`eme fondamental de solutions $s_{1}$, $\ldots$ , $s_{p}$  ( $s_{k} = ( s_{k}^{l} )$, $l = 1 , \ldots , p$), les $s_{k}^{l}$  \'etant des germes de fonctions analytiques (m\^emes globales) \`a l'origine de $\mathbb{R}$. Il est clair  que les solutions formelles $\hat{g} = ( \hat{g}_{1} , \ldots , \hat{g}_{p} )$ s'\'ecrivent sous la forme:
$$\hat{g} ( x_{1} , \ldots , x_{n} ) = \sum \hat{h}^{j}_{k} ( x_{2} , \ldots , x_{n} ) . s_{j} ( x_{1} )$$
o\`u les $\hat{h}^{j}_{k} \in \hat{\mathcal{E}}_{n - 1}$. Soient $h^{j}_{k} \in \mathcal{E}_{n - 1}$ des r\'ealisations de Borel des $\hat{h}^{j}_{k}$, les composantes   $g_{1}$, $\ldots$ , $g_{p}$ du vecteur $g ( x_{1} , \ldots , x_{n} ) = \sum h^{j}_{k} ( x_{2} , \ldots , x_{n} ) . s_{j} ( x_{1} )$ sont $\mathcal{C}^{\infty}$ et solution du syst\`eme $\mathcal{D}(\frac{\partial}{\partial x_{1}})$. Une v\'erification \'el\'ementaire montre que
$\mathcal{L}^{'} = \langle \frac{\partial}{\partial x_{1}} , g_{1} \frac{\partial}{\partial x_{1}} \ldots , g_{p} \frac{\partial}{\partial x_{1}} \rangle$ est une alg\`ebre de Lie dont le jet Taylorien infini est $\hat{\Phi}_{\ast} \hat{\mathcal{L}}$. En consid\'erant une r\'ealisation $\mathcal{C}^{\infty}$ de $\hat{\Phi}$
on obtient trivialement une alg\`ebre de Lie, $\mathcal{L} = \Phi^{- 1}_{\ast} \mathcal{L}^{'}$, telle que $T_{\underline{0}} :  \mathcal{L} \rightarrow \hat{\mathcal{L}}$ soit un isomorphisme. D'o\`u la:

\begin{proposition} Soit $\hat{\mathcal{L}} = E . \hat{X} \subset \hat{\mathcal{X}}_{n}$ une alg\`ebre saturable de rang $1$. Si $\hat{X}$ est non singulier, alors il existe une sous alg\`ebre $\mathcal{L} \subset \mathcal{X}_{n}$  telle que $T_{\underline{0}} :  \mathcal{L} \rightarrow \hat{\mathcal{L}}$ est un isomorphisme.
\end{proposition}

\subsubsection{Alg\`ebres saturables de rang 1 \`a 1-jet nul}

Ici $J^{1} \hat{X}$ est identiquement nul. Dans ce cas pour tout \'el\'ement $\hat{Z}$ de $\hat{\mathcal{L}}$ le jet d'ordre 1, $J^{1} \hat{Z}$, de $\hat{Z}$ est nul.
Ceci implique que l'application adjointe $ad_{\hat{Z}} : \hat{\mathcal{L}} \rightarrow \hat{\mathcal{L}}$ est nilpotente. En effet les valeurs propres de  $ad_{\hat{Z}}$ sont n\'ecessairement nulles. Par suite l'alg\`ebre de Lie $\hat{\mathcal{L}}$ est nilpotente et poss\`ede donc un centre non trivial $\hat{\mathcal{C}} \subset \hat{\mathcal{L}} =  \langle \hat{X} , \hat{f}_{1} \hat{X} , \ldots , \hat{f}_{p} \hat{X} \rangle$. Si $\hat{X}$ est dans le centre, alors $[\hat{X} , \hat{f}_{i} \hat{X} ] = 0$ pour tout $i$, et donc $\hat{X} (f_{i}) = 0$, pour $i = 1 , \ldots, p$. Il en r\'esulte que $\hat{\mathcal{L}}$ est ab\'elienne et que $\hat{X}$ poss\`ede une int\'egrale premi\`ere non constante si $\dim \hat{\mathcal{L}} \geq 2$. Si $\hat{X}$ n'est pas dans $\hat{\mathcal{C}}$, alors il existe un \'el\'ement non constant de $\hat{\mathcal{E}}_{n}$ tel que $\hat{f} \hat{ X}$ appartient \`a $\hat{\mathcal{C}}$. Comme $\hat{f} \hat{ X}$ et $\hat{ X}$ commutent, $\hat{f}$ est une int\'egrale premi\`ere non constante, par suite de $0 =  [\hat{f} \hat{X} , \hat{f}_{i} \hat{X}] = \hat{X} (\hat{f}_{i}). \hat{X}$ on d\'eduit que chaque $\hat{f}_{i}$ est une int\'egrale premi\`ere de $\hat{X}$. Ceci conduit au fait que $\hat{X}$ appartient \`a $\hat{\mathcal{C}}$ en contradiction avec l'hypoth\`ese. Donc $\hat{\mathcal{L}}$ est ab\'elienne et on a la:

\begin{proposition} Soit $\hat{\mathcal{L}} = E . \hat{X} \subset \hat{\mathcal{X}}_{n}$ une alg\`ebre saturable de rang $1$ et de dimension finie. Si $J^{1} \hat{X} = 0$ et $\dim \hat{\mathcal{L}} > 1$, alors $\hat{\mathcal{L}}$ est ab\'elienne et le champ $\hat{X}$ poss\`ede une int\'egrale premi\`ere non constante.
\end{proposition}

\begin{remark} Si le champ poss\`ede une int\'egrale premi\`ere non constante $\hat{f}_{0}$ on peut supposer que $\hat{f}_{0}(0) = 0$ dans le cas saturable. Chaque \'el\'ement $\hat{l}$ de $\hat{\mathcal{E}}_{1}$ produit une int\'egrale premi\`ere $\hat{l}(\hat{f}_{0})$. En particulier l'espace vectoriel $\hat{\mathcal{L}} = \{ \hat{l}(\hat{f}_{0}) \hat{X} / \hat{l} \in \hat{\mathcal{E}}_{1} \}$ est une alg\`ebre de Lie ab\'elienne de dimension infinie.
\end{remark}

Nous allons traiter le cas sp\'ecifique, en  dimension d'espace $2$, des alg\`ebres saturables. L'avantage de la dimension $2$ est cons\'equence d'un r\'esultat de J.-F. Mattei et R. Moussu \cite{Mattei}  qui dit que si  $\hat{X}$ est un champ non nul, appartenant \`a $\hat{\mathcal{X}}_{2}$, ayant une int\'egrale premi\`ere formelle non constante, alors l'anneau des ses int\'egrales premi\`eres formelles $\hat{\mathcal{A}} ( \hat{X} ) \subset \hat{\mathcal{E}}_{2}$ est engendr\'e par un \'el\'ement. Plus pr\'ecisement il existe $\hat{f}_{0}$
dans $\hat{\mathcal{E}}_{2}$, dit minimal, d\'efini \`a composition \`a gauche pr\`es par les \'el\'ements de $\widehat{\mathrm{Diff}}_{1} (\mathbb{R}_{0})$, tel que: $$\hat{\mathcal{A}} ( \hat{X} ) = \{ \hat{l} ( \hat{f}_{0} ) / \hat{l} \in \hat{\mathcal{E}}_{1} \}$$

Cet \'enonc\'e est \'etabli dans \cite{Mattei} dans les cadres complexes formel et holomorphe, mais il s'adapte assez facilement au cas r\'eel. On se propose d'\'etablir la:

\begin{proposition} Soit $\hat{\mathcal{L}} = E . \hat{X} \subset \hat{\mathcal{X}}_{2}$ une sous alg\`ebre saturable de rang $1$ et $\dim \hat{\mathcal{L}} > 1$. On suppose que  $\hat{\mathcal{L}}$ est ab\'elienne et que le champ $\hat{X}$ poss\`ede une int\'egrale premi\`ere non constante. Alors  il existe une sous alg\`ebre $\mathcal{L} \subset \mathcal{X}_{2}$  telle que $T_{\underline{0}} :  \mathcal{L} \rightarrow \hat{\mathcal{L}}$ soit un isomorphisme.
\end{proposition}

\begin{proof}
Soit $\hat{f}_{0}$ une int\'egrale premi\`ere minimale de $\hat{X}$. La d\'ecomposition complexe en facteurs irr\'eductibles de $\hat{f}_{0}$ est du type:

$$\hat{f}_{0} = \hat{f}_{1}^{n_{1}} \ldots \hat{f}_{k}^{n_{k}} (\hat{f}_{k + 1} \overline{\hat{f}}_{k + 1} )^{n_{k + 1}} \ldots (\hat{f}_{p} \overline{\hat{f}}_{p} )^{n_{p}}$$
 avec $n_{i} \in \mathbb{N}$, $\hat{f}_{i} \in \hat{\mathcal{E}}_{2}$ pour $i = 1 , \ldots , k$ et $\hat{f}_{i} \in \hat{\mathcal{O}}_{2} \setminus  \hat{\mathcal{E}}_{2}$ pour $i = k + 1 , \ldots , p$ o\`u $\hat{\mathcal{O}}_{n}$ d\'esigne l'anneau des s\'eries formelles en $n$ variables complexes et $\overline{\hat{f}}_{i}$  le conjugu\'e de $\hat{f}_{i}$. Notons $\hat{\omega}_{0}$ la 1-forme d\'efinie par:
$$\hat{\omega}_{0} = \hat{f}_{1}^{n_{1}} \ldots \hat{f}_{k}^{n_{k}} \hat{f}_{k + 1} \overline{\hat{f}}_{k + 1}  \ldots \hat{f}_{p} \overline{\hat{f}}_{p} ( \sum_{i = 1}^{k} n_{i}\frac{d \hat{f}_{i}}{\hat{f}_{i}} + \sum_{i = k + 1}^{p} n_{i} (\frac{d \hat{f}_{i}\overline{\hat{f}}_{i}}{\hat{f}_{i}\overline{\hat{f}}_{i}} )).$$

Cette 1-forme, \`a priori complexe, est visiblement r\'eelle et \`a singularit\'e alg\'ebriquement isol\'ee. Soit $\hat{X}_{0}$ un champ de vecteurs dual de $\hat{\omega}_{0}$: $\hat{\omega}_{0} = i_{\hat{X}_{0}} d x_{1} \wedge d x_{2}$. L'alg\`ebre de Lie $\hat{\mathcal{L}} = E. \hat{X}$ est du type $\hat{\mathcal{L}} = \langle \hat{X} , \hat{l}_{1}(\hat{f}_{0}) \hat{X} , \ldots , \hat{l}_{p}(\hat{f}_{0}) \hat{X} \rangle$. Consid\'erons des r\'ealisations $\mathcal{C}^{\infty}$ de $\hat{f}_{1} , \ldots , \hat{f}_{k} , \hat{g}_{k + 1} = \hat{f}_{k + 1} \overline{\hat{f}}_{k + 1} , \ldots , \hat{g}_{p} = \hat{f}_{p} \overline{\hat{f}}_{p}$ not\'ees respectivement $f_{1} , \ldots , f_{k} , g_{k + 1} , \ldots , g_{p}$. Ces r\'ealisations induisent une r\'ealisation $\mathcal{C}^{\infty}$ de $\hat{X}_{0}$, not\'ee $X_{0}$, poss\'edant l'int\'egrale premi\`ere \\
 $f_{0} = f_{1}^{n_{1}} \ldots f_{k}^{n_{k}} g_{k + 1}^{n_{k + 1}} \ldots g_{p}^{n_{p}}$. Comme $\hat{\omega}_{0}$ est \`a singularit\'e isol\'ee il existe $\hat{h} \in \hat{\mathcal{E}}_{2}$ tel que $\hat{X} = \hat{h} \hat{X}_{0}$. Le choix d'une r\'ealisation $h$ de $\hat{h}$ en produit une, $X = h X_{0}$, pour $\hat{X}$. Consid\'erons des r\'ealisations $l_{1} , \ldots , l_{p}$ de $\hat{l}_{1} , \ldots , \hat{l}_{p}$ respectivement. L'alg\`ebre de Lie $\mathcal{L} = \langle X , l_{1}(f_{0}) X , \ldots , l_{p}(f_{0}) X \rangle$ est une r\'ealisation ab\'elienne de $\hat{\mathcal{L}}$ telle que $T_{\underline{0}}: \mathcal{L} \rightarrow \hat{\mathcal{L}}$ est un isomorphisme. 
\end{proof}

\subsubsection{Alg\`ebres saturables de rang 1 et \`a 1-jet non  nul}

On \'ecrit $\hat{\mathcal{L}} = \langle \hat{X} , \hat{f}_{1} \hat{X} , \ldots , \hat{f}_{p} \hat{X}_{p} \rangle$ avec $\hat{X}(0) = 0$ et $J^{1} \hat{X} \neq 0$. Ici encore l'espace $E = \langle  1 , \hat{f}_{1} , \ldots , \hat{f}_{p} \rangle$ est invariant par la d\'erivation $\hat{X}$. On se place  en dimension 2; nous distinguons les diff\'erents types de Jordan pour le 1-jet $X_{1}$ de $\hat{X}$ \`a conjugaison lin\'eaire pr\`es:
\begin{enumerate}
\item  $\lambda_{1} x_{1}\frac{\partial}{\partial x_{1}} + \lambda_{2} x_{1}\frac{\partial}{\partial x_{2}}$ avec $(\lambda_{1} , \lambda_{2}) \neq (0 , 0)$ (cas diagonal r\'eel).
\item $(\alpha x_{1} - \beta x_{2})\frac{\partial}{\partial x_{1}} + (\beta x_{1} + \alpha x_{2}) \frac{\partial}{\partial x_{2}}$ o\`u $(\alpha , \beta) \neq (0 , 0)$ (cas diagonal complexe).
\item $(\lambda x_{1} + x_{2})\frac{\partial}{\partial x_{1}} + \lambda x_{2} \frac{\partial}{\partial x_{2}}$ (cas non semi-simple).
\end{enumerate}
Le champ $\hat{X}$ est formellement conjugu\'e \`a sa partie lin\'eaire dans les cas suivants:
\begin{itemize}
\item Cas 1 sans r\'esonnance, i.e. $i_{1} \lambda_{1} + i_{2} \lambda_{2} \neq \lambda_{j}$ pour tout couple $(i_{1}, i_{2})$ d'entiers, $i_{1} + i_{2} \geq 2$ et $j \in \{ 1 , 2 \}$.
\item Cas 2 avec $\alpha \neq 0$. On observe ici que les valeurs propres $\alpha \pm i \beta$ sont complexes et sans r\'esonnances lorsque $\beta \neq 0$.
\item Cas 3 avec $\lambda \neq 0$.
\end{itemize}

{\bf 1.2.3.1 R\'ealisation $\mathcal{C}^{\infty}$ dans les cas semi-simples non r\'esonnants}\\

Ils correspondent aux cas diagonal r\'eel sans r\'esonnance et  diagonal complexe avec $\alpha \neq 0$. Dans le premier cas, \`a conjugaison formelle pr\`es, on suppose que $\hat{X} = X_{1} = J^{1} \hat{X}$. Le champ $X_{1}$ \'etant semi-simple, il l'est en tant que d\'erivation et sa restriction \`a $E$ l'est aussi. On peut donc supposer que les $\hat{f}_{j}$ forment une base de vecteurs propres de $\hat{X}$: $\hat{X} \hat{f}_{j} = \mu_{j} \hat{f}_{j}$, $\mu_{j} \in \mathbb{R}$.

Si $x_{1}^{i_{1}} x_{2}^{i_{2}}$ est un mon\^ome apparaissant avec un coefficient non nul dans $\hat{f}_{j}$, alors on a:

\begin{equation}
i_{1} \lambda_{1} + i_{2} \lambda_{2} = \mu_{j}.
\end{equation}

\begin{lemma} Si $\lambda_{1}$ et $\lambda_{2}$ sont non r\'esonnants, l'ensemble $\Lambda$ des couples $(i_{1}, i_{2}) \in \mathbb{N}^{2}$ satisfaisant $(3)$ est fini.
\end{lemma}

\begin{proof} 
On fixe $(k_{1} , k_{2}) \in \Lambda$ r\'ealisant le minimum pour l'ordre lexicographique. Si $(i_{1} , i_{2})$ appartient \`a $\Lambda$ et $(i_{1} , i_{2}) \neq (k_{1} , k_{2})$ on a:
\begin{equation}
(i_{1} - k_{1}) \lambda_{1} + (i_{2} - k_{2}) \lambda_{2} = 0.
\end{equation}

Comme $(\lambda_{1} , \lambda_{2})$ est un couple non r\'esonnant alors $i_{1} \neq k_{1}$ et $i_{2} \neq k_{2}$. Par suite $i_{1} > k_{1}$ et $i_{2} \neq k_{2}$. Si $i_{2}$ \'etait strictement sup\'erieur \`a $k_{2}$, alors $(i_{1} - k_{1})$ et $(i_{2} - k_{2})$ seraient positifs; ce qui cr\'eerait une r\'esonnance. Donc $i_{2}$ est strictement inf\'erieur \`a $k_{2}$ et ne peut donc prendre qu'un nombre fini de valeurs. Puisque $i_{1}$ est d\'etermin\'e par $i_{2}$ le lemme est v\'erifi\'e.
\end{proof}

Une cons\'equence du lemme est que chaque $\hat{f}_{j}$ est un polyn\^ome. \\

Dans le cas diagonal complexe lin\'eaire avec $\alpha \neq 0$ \`a conjugaison formelle pr\`es on suppose que $\hat{X}$ est \'egal \`a son 1-jet  et on obtient un r\'esultat
analogue au cas r\'eel.  Ainsi l'alg\`ebre $\hat{\mathcal{L}}$ est conjugu\'ee \`a une alg\`ebre $\mathcal{L}_{pol}$ de champs de vecteurs polynomiaux dans les cas diagonal r\'eel non r\'esonnant et diagonal complexe avec $\alpha \neq 0$:

$$\hat{\mathcal{L}} = \hat{\phi}_{\ast} \mathcal{L}_{pol} , \ \hat{\phi} \in \widehat{\mathrm{Diff}}(\mathbb{R}^{2}_{0}) \ {\rm et} \ \mathcal{L}_{pol} \subset \mathcal{X}_{2}. $$

Si $\phi \in \mathrm{Diff}(\mathbb{R}^{2}_{0})$ est une r\'ealisation $\mathcal{C}^{\infty}$ de $\hat{\phi}$ alors $\mathcal{L} = \phi_{\ast} \mathcal{L}_{pol}$ est une r\'ealisation de Borel de $\hat{\mathcal{L}}$ telle que $T_{\underline{0}} : \mathcal{L} \rightarrow \hat{\mathcal{L}}$ est un isomorphisme.\\

{\bf 1.2.3.2 Cas r\'esonnant}

Nous \'etudions d'abord les r\'esonnances de type Poincar\'e-Dulac, cas o\`u le 1-jet de $\hat{X}$ s'\'ecrit $X_{1} = \lambda (x_{1} \frac{\partial}{\partial x_{1}} + n x_{2} \frac{\partial}{\partial x_{2}})$, $\lambda \neq 0$ et $n \in \mathbb{N} \setminus \{0 , 1 \}$. Ici il y a une seule r\'esonnance: $n \lambda_{1} = \lambda_{2}$.
D'apr\`es le Th\'eor\`eme de Poincar\'e-Dulac, \cite{Cerveau2} le champ $\hat{X}$ est \`a conjugaison formelle pr\`es $\hat{X} =  \lambda [ x_{1} \frac{\partial}{\partial x_{1}} + ( n x_{2} + \mu x_{1}^{n}) \frac{\partial}{\partial x_{2}} ]$ avec $\mu = 0$ ou $1$. Lorsque $\mu = 0$ on est dans le cas diagonal r\'eel. Ce cas se traite comme en 1.2.3.1: $\hat{\mathcal{L}}$ est formellement conjugu\'ee \`a une alg\`ebre polyn\^omiale. Ici encore $\hat{\mathcal{L}}$ se r\'ealise de fa\c{c}on $\mathcal{C}^{\infty}$.\\

Sinon on a une ramification du cas $n = 1 , \mu= 1$ qu'on fera en fin de paragraphe.\\

Nous avons ensuite la situation des r\'esonnances pures qui se traite en deux sous cas:
\begin{itemize}
\item Cas hyperbolique o\`u le 1-jet $X_{1} = \lambda (q x_{1} \frac{\partial}{\partial x_{1}} - p  x_{2} \frac{\partial}{\partial x_{2}})$, $\lambda \neq 0$, $p$ et $q$ sont des entiers positifs et $\langle p , q \rangle = 1$ ou  $p = 0$ et $q \neq 0$ ou cas noeud-col $p = 0$ et $q = 1$.
\item Cas elliptique:  $X_{1} = \beta (- x_{2} \frac{\partial}{\partial x_{1}} +  x_{1} \frac{\partial}{\partial x_{2}})$, o\`u $\beta \neq 0$.
\end{itemize}

 Notons que le champ $\lambda (q x_{1} \frac{\partial}{\partial x_{1}} - p x_{2} \frac{\partial}{\partial x_{2}})$ a une int\'egrale premi\`ere monomiale $x_{1}^{p} x_{2}^{q}$ tandis que le champ $\beta ( - x_{2}  \frac{\partial}{\partial x_{1}} + x_{1} \frac{\partial}{\partial x_{2}} )$ a l'int\'egrale premi\`ere $x_{1}^{2} + x_{2}^{2}$. Nous allons d\'etailler le premier cas, le second pr\'esente une certaine similarit\'e. La th\'eorie des formes normales J. Martinet \cite{Martinet} (ou de la Jordanisation) permet d'\'ecrire $\hat{X} = \hat{S} + \hat{N}$ avec $[\hat{S} , \hat{N} ] = 0$ et $\hat{S}$ est formellement conjugu\'e \`a sa partie lin\'eaire, i.e. $\hat{\Phi}_{\ast} \hat{S} = \lambda (q x_{1} \frac{\partial}{\partial x_{1}} - p x_{2} \frac{\partial}{\partial x_{2}})$, et le champ $\hat{N}$ est nilpotent, ce qui signifie ici que $J^{1} \hat{N} = 0$. Remarquons que l'on peut supposer pour notre contexte que $\lambda = 1$.
Comme $E = \langle 1, \hat{f}_{1} , \ldots , \hat{f}_{p} \rangle$ est invariant sous l'action de $\hat{X}$, il l'est sous l'action de sa partie semi simple $\hat{S}$  la restriction de $\hat{S}$ \`a $E$ reste semi simple, c'est \`a dire diagonalisable: \`a l'action pr\`es de $\hat{\phi}$ que $\hat{S} = X_{1}$. D'o\`u:

\begin{equation}
(q x_{1} \frac{\partial}{\partial x_{1}} - p x_{2} \frac{\partial}{\partial x_{2}}) ( \hat{f}_{j} ) = \mu_{j} \hat{f}_{j}.
\end{equation}

Si $\mu_{j} = 0$, $\hat{f}_{j}$ est une int\'egrale premi\`ere de $\hat{S}$ et s'\'ecrit $\hat{f}_{j} = \hat{l}_{j}(x_{1}^{p} x_{2}^{q})$ avec $\hat{l}_{j} \in \hat{\mathcal{E}}_{1}$. Si $\mu_{j} \neq 0 $, alors la relation $(5)$ implique que $\mu_{j}$ est de la forme $\mu_{j} = q r_{j} - p s_{j}$. Soit $F_{j} = \{ (k , l) \in \mathbb{N}^{2} \setminus (0 , 0) / kq - p l = \mu_{j} \}$ et $(r_{j} , s_{j})$ le plus petit \'el\'ement de $F_{j}$ pour l'ordre lexicographique. On a un lemme technique \'el\'ementaire suivant qui est analogue au Lemme 10:

\begin{lemma} Si $(k , l)$ appartient \`a $F_{j}$ alors $(k , l) = (r_{j} , s_{j}) + s (p , q)$, o\`u $s \in \mathbb{N}$.
\end{lemma}

\begin{proof}
Si $(p , q)$ est de type $(1 , 0)$, alors $(r_{j}, s_{j})$ est de type $(0 , s_{j})$ et $(k , l) = (0 , s_{j}) + k (1 , 0)$. Supposons dor\'enavant  que $p q \neq 0$. Si $(k , l ) \neq (r_{j}, s_{j})$ on a $(k - r_{j}) q - (l - s_{j}) p = 0$ et donc $k - r_{j} \neq 0$. Par suite, puisque $(r_{j}, s_{j})$ est minimal pour l'ordre lexicographique et $p q \neq 0$, $k - r_{j} > 0$ et n\'ecessairement $l - s_{j} > 0$. Le fait que $\langle p , q \rangle = 1$ implique
l'existence de $s$, d'o\`u le lemme. 
\end{proof}

Il r\'esulte du lemme pr\'ec\'edent que l'on peut \'ecrire $\hat{f}_{j} = x_{1}^{r_{j}} x_{2}^{s_{j}} \hat{\varphi} (x_{1}^{p} x_{2}^{q})$ o\`u $\hat{\varphi}_{j} \in \hat{\mathcal{E}}_{1}$, $r_{j}$ et $s_{j}$ \'etant des entiers comme ci-dessus. Le champ $\hat{N}$ commutant avec $\hat{S} = \lambda (q x_{1} \frac{\partial}{\partial x_{1}} - p x_{2} \frac{\partial}{\partial x_{2}})$ a la forme suivante:

 $$\hat{N} = x_{1} \hat{\alpha} (x_{1}^{p} x_{2}^{q}) \frac{\partial}{\partial x_{1}} + x_{2} \hat{\beta} (x_{1}^{p} x_{2}^{q}) \frac{\partial}{\partial x_{2}}$$
 o\`u $\hat{\alpha}$ et $\hat{\beta}$ appartiennent \`a $\hat{\mathcal{E}}_{1}$.\\

En fait en remarquant que les transformations $(x_{1} \hat{A} (x_{1}^{p} x_{2}^{q}) , x_{2} \hat{B} (x_{1}^{p} x_{2}^{q}))$, $\hat{A} , \hat{B} \in \hat{\mathcal{E}}_{1}$, laissent invariant $\hat{S}$ on peut supposer que $\hat{N}$ est de la forme:
$$\hat{N} = \hat{a} (x_{1}^{p} x_{2}^{q}) (\lambda_{1} x_{1} \frac{\partial}{\partial x_{1}} + \lambda_{2} x_{2} \frac{\partial}{\partial x_{2}})$$

o\`u  $\lambda_{1} , \lambda_{2} \in \mathbb{R}$ et $\hat{a} \in \hat{\mathcal{E}}_{1}$. Comme le champ $\hat{X} = \hat{S} + \hat{N}$ agissant sur $E$ garde au moins un vecteur propre, disons $\hat{f}_{1}$, on a: $$\hat{X} ( x_{1}^{r_{j}} x_{2}^{s_{j}} \hat{\varphi}_{j} ( x_{1}^{p} x_{2}^{q})) = ( q r_{1} - p s_{1} ) x_{1}^{r_{1}} x_{2}^{s_{1}} \hat{\varphi}_{1} ( x_{1}^{p} x_{2}^{q}).$$

Puisque $\hat{a} \neq 0$ on obtient par un calcul \'el\'ementaire:

\begin{equation}
(r_{1}\lambda_{1} + s_{1} \lambda_{2}) \hat{\varphi}_{1} ( t ) +  (p\lambda_{1} + q \lambda_{2} ) t \hat{\varphi}_{1}^{'} ( t )  = 0.
\end{equation}
Remarquons la possibilit\'e des cas sp\'eciaux suivants:
\begin{itemize}
\item $(\lambda_{1} , \lambda_{2}) = (0 ,0)$; auquel cas $\hat{X} = \hat{S}$ et $(6)$ ne donne aucun renseignement sur $\hat{\varphi}_{1}$.
\item $(\lambda_{1} , \lambda_{2}) \neq (0 ,0)$ et $p\lambda_{1} + q \lambda_{2} = 0$; comme  $\hat{\varphi}_{1}$ est non identiquement nul, $r_{1} q - s_{1} p = 0$ et $\hat{f}_{1} = \hat{\psi}( x_{1}^{p} x_{2}^{q}))$ est une int\'egrale premi\`ere de $\hat{X}$ qui s'\'ecrit dans ce cas $\hat{X} = \hat{b} ( x_{1}^{p} x_{2}^{q})) ( q x_{1} \frac{\partial}{\partial x_{1}} -p x_{2} \frac{\partial}{\partial x_{2}})$.
\end{itemize}

Dans le cas g\'en\'erique o\`u $(p \lambda_{1} + q \lambda_{2}) \neq 0$ on constate que  $\hat{\varphi}_{1}$ est un mon\^ome: $\hat{\varphi}_{1} = \varepsilon t^{s}$  avec $s = \frac{r_{1}\lambda_{1} + s_{1} \lambda_{2}}{p\lambda_{1} + q \lambda_{2}} \in \mathbb{N}$; ce qui montre \'egalement que $\hat{f}_{1}$ est mon\^omiale.

\begin{lemma} La restriction $\hat{X}_{ \mid E} : E \rightarrow E$ de la d\'erivation $\hat{X}$ est semi simple.
\end{lemma}

\begin{proof}
Si ce n'est pas le cas, comme les valeurs propres de $\hat{X}$ sont celles de $\hat{S}$, et donc r\'eelles sous nos hypoth\`eses, la Jordanisation de $\hat{X}_{ \mid E}$ est r\'eelle et il existe, \`a re-indexation pr\`es, $\hat{f}_{1}$ et $\hat{f}_{2 }$ tels que $\hat{X} (\hat{f}_{1} ) = \mu_{1} \hat{f}_{1} $ et $\hat{X} (\hat{f}_{2} ) = \mu_{1} \hat{f}_{2}  + \hat{f}_{1}$. Il en r\'esulte que:
$$[ \hat{f}_{1} \hat{X} , \hat{f}_{2} \hat{X}] = ( \hat{f}_{1} \hat{X} (\hat{f}_{2} ) - \hat{f}_{2} \hat{X}( \hat{f}_{1} ) ) \hat{X} = \hat{f}^{2}_{1}\hat{X}.$$
On en d\'eduit que $\hat{f}_{1}^{2} \in E$; de m\^eme on a les \'equations:

\begin{equation}
[ \hat{f}_{1} \hat{X} , \hat{f}_{1}^{2} \hat{X}] = \mu_{1} \hat{f}_{1}^{3} \hat{X}
\end{equation}
\begin{equation}[ \hat{f}_{1}^{2} \hat{X} , \hat{f}_{2} \hat{X}] = (\hat{f}_{1}^{2} \hat{X}(\hat{f}_{2} ) - \hat{f}_{2} \hat{X} (\hat{f}_{1}^{2})) \hat{X} = ( \hat{f}_{1}^{3} - \mu_{1} \hat{f}_{1}^{2} \hat{f}_{2} ) \hat{X}
\end{equation}
On en d\'eduit que si $\mu_{1}$ est non nul, d'apr\`es $(7)$, $\hat{f}_{1}^{3}$ appartient \`a $E$; dans le cas contraire $\hat{f}_{1}^{3}$ appartient \'egalement \`a $E$ d'apr\`es $(8)$. Supposons par induction que $\hat{f}_{1}^{k}$ appartient \`a $E$, on a alors $[ \hat{f}_{1} \hat{X} , \hat{f}_{1}^{k} \hat{X}] = \mu_{1} (k - 1 )\hat{f}_{1}^{k + 1} \hat{X}$. Et de nouveau $\hat{f}_{1}^{k + 1}$ appartient \`a $E$ lorsque $\mu_{1}$ est non nul. Si $\mu_{1}$ est nul on a $[ \hat{f}_{1}^{k} \hat{X} , \hat{f}_{2} \hat{X}] = \hat{f}_{1}^{k + 1} \hat{X}$, et on en d\'eduit encore que $\hat{f}_{1}^{k + 1}$ appartient \`a $E$.  Comme $\hat{f}_{1} ( 0 ) = 0$, les ordres des $\hat{f}_{1}^{k }$ augmentent; ce qui implique que $E$ est dimension infinie. Ceci est en contradiction avec $\dim E < + \infty$.
\end{proof}

Il  r\'esulte du Lemme 11  que $\hat{\mathcal{L}} = \{ \hat{X} , \hat{f}_{1} \hat{X} ,  \ldots , \hat{f}_{p} \hat{X} \}$, o\`u les  $\hat{f}_{i}$ sont des vecteurs propres de $\hat{X}$: $\hat{X} ( \hat{f}_{i} ) = \mu_{i} \hat{f}_{i}$, $\mu_{i} \in \mathbb{R}$. La structure d'alg\`ebre de Lie est donn\'ee par $[  \hat{X} , \hat{f}_{i} \hat{X}] = \mu_{i} \hat{f}_{i} \hat{X}$ et $[ \hat{f}_{i} \hat{X} , \hat{f}_{j} \hat{X}] = ( \mu_{i}- \mu_{j} ) \hat{f}_{i} \hat{f}_{j} \hat{X}$.

Comme $f_{0} = 1$ appartient \`a $E$ l'un des $\mu_{i}$, disons $\mu_{0}$, est nul. Si tous les $\mu_{i}$ sont nuls, alors tous les $\hat{f}_{i}$ sont des int\'egrales premi\`eres du champ $\hat{X}$; comme $E$ est suppos\'e de dimension superieur ou \'egal \`a deux l'une au moins de ces int\'egrales premi\`eres est non constante. Le champ $\hat{X}$ s'\'ecrit $\hat{X} = \hat{b} ( x_{1}^{p} x_{2}^{q}) (q x_{1} \frac{\partial}{\partial x_{1}} - p x_{2} \frac{\partial}{\partial x_{2}})$ et  $\hat{f}_{i}$ est \'egale \`a $\hat{l}_{i} ( x_{1}^{p} x_{2}^{q})$, o\`u $\hat{b}$ et les $\hat{l}_{i}$ appartiennent \`a $\hat{\mathcal{E}}_{1}$. Soit $\hat{\Phi}$ le diff\'eomorphisme de mise sous forme normale qui lin\'earise $\hat{S}$. En choisissant des r\'ealisations $\Phi$, $b$ et $f_{i}$ de $\hat{\Phi}$, $\hat{b}$ et $\hat{f}_{i}$ respectivement on obtient une r\'ealisation $\mathcal{L}$ de $\hat{\mathcal{L}}$ telle que $T_{\underline{0}} : \mathcal{L} \rightarrow \hat{\mathcal{L}}$ est un isomorphisme.

Supposons  les $\mu_{i}$ sont non tous nuls. Nous devons envisager deux cas: celui o\`u tous les $\mu_{i}$ sont non nuls pour $i \neq 0$, et celui o\`u l'un des $\mu_{i}$ est nul pour $i \neq 0$. Pla\c{c}ons nous dans ce dernier cas. Disons $\hat{X}\hat{f}_{1} = 0$, $\hat{f}_{1}$ est non constant et $\hat{X}\hat{f}_{2} = \mu_{2} \hat{f}_{2}$ avec $\mu_{2} \neq 0$. On peut supposer que $\hat{f}_{1}(0) = 0$; ce que l'on fera. On a
$[ \hat{f}_{1} \hat{X} , \hat{f}_{2} \hat{X}] = \mu_{2} \hat{f}_{1} \hat{f}_{2} \hat{X}$. Et donc  $\hat{f}_{1} \hat{f}_{2}^{2} \hat{X}$ appartient \`a $\hat{\mathcal{L}}$. De m\^eme $[ \hat{f}_{1} \hat{X} , \hat{f}_{1} \hat{f}_{2} \hat{X}] = \mu_{2} \hat{f}_{1}^{2} \hat{f}_{2} \hat{X}$ et donc $\hat{f}_{1}^{2} \hat{f}_{2} \hat{X}$ appartient \`a $\hat{\mathcal{L}}$. Par induction on montre que tous les $\hat{f}_{1}^{k} \hat{f}_{2} \hat{X}$, $k \in \mathbb{N}$, sont dans $\hat{\mathcal{L}}$. Comme $\hat{f}_{1} (0) = 0$ ils forment une famille libre en contradiction avec la finitude de la dimension de $\hat{\mathcal{L}}$. Ainsi ce cas ne se pr\'esente pas.\\

 Consid\'erons  la situation o\`u tous les $\mu_{i}$, \`a l'exception de $\mu_{0}$, sont non nuls, i.e. la dimension du sous espace propre $V(0)$ associ\'e \`a $\mu_{0}$ est $1$. Les $\hat{f}_{i}$ \'etant propres pour $\hat{X}$, ils le sont pour sa partie semi-simple $\hat{S}$ et sont annul\'es par la partie nilpotente: $\hat{S} \hat{f}_{i} = \mu_{i} \hat{f}_{i}$ et $\hat{N} \hat{f}_{i} = 0$.
Supposons l'existence de $\mu_{i} \neq \mu_{j}$ pour deux indices distincts $i$ et $j$ qu'on suppose \^etre $1$ et $2$. On a alors
$[ \hat{f}_{1} \hat{X} , \hat{f}_{2} \hat{X}] = ( \mu_{2} - \mu_{1}) \hat{f}_{1} \hat{f}_{2} \hat{X}$ et donc $\hat{f}_{1} \hat{f}_{2}. \hat{X}$ est dans $\hat{\mathcal{L}}$. On v\'erifie que $\hat{X} (\hat{f}_{1} \hat{f}_{2}) = ( \mu_{2} + \mu_{1}) \hat{f}_{1} \hat{f}_{2}$. Notons que $\mu_{2} + \mu_{1}$ est non nul puisque $V(0)$ est de dimension r\'eelle $1$. On a aussi $[ \hat{f}_{1} \hat{X} , \hat{f}_{1} \hat{f}_{2} \hat{X}] =  \mu_{2}  \hat{f}_{1}^{2} \hat{f}_{2} \hat{X}$. Comme $\mu_{2}$ est non nul, $\hat{f}_{1}^{2}\hat{f}_{2} \hat{X}$ appartient \`a $\hat{\mathcal{L}}$. Supposons que pour $n \leq k$, $\hat{f}^{k}_{1} \hat{f}_{2} \hat{X}$ soit dans $\hat{\mathcal{L}}$. Des relations  $\hat{X} (\hat{f}^{n}_{1} \hat{f}_{2}) = ( n \mu_{1} + \mu_{2}) \hat{f}^{n}_{1} \hat{f}_{2}$ et  $\dim V(0) = 1$ on voit que n\'ecessairement $n \mu_{1} + \mu_{2} \neq 0$. On d\'eduit alors de $[ \hat{f}_{1} \hat{X} , \hat{f}_{1}^{k} \hat{f}_{2} \hat{X}] = ( (k - 1)\mu_{1} + \mu_{2} ) \hat{f}_{1}^{k + 1} \hat{f}_{2} \hat{X}$ que $\hat{f}^{k + 1}_{1} \hat{f}_{2} \hat{X}$ est dans $\hat{\mathcal{L}}$. Ainsi $\hat{f}^{k}_{1} \hat{f}_{2} \hat{X}$ appartient \`a $\hat{\mathcal{L}}$ et ceci pour tout $k$. Ce qui contredit la finitude de la dimension de $E$.

Donc tous les r\'eels $\mu_{i}$, hormis $\mu_{0}$, sont \'egaux \`a une constante $\mu$ non nulle. On en d\'eduit que les $\hat{f}_{i}$ sont du type $\hat{f}_{i} = x_{1}^{r} x_{2}^{s}\hat{\varphi}_{i} (x_{1}^{p} x_{2}^{q})$. L'alg\`ebre $\hat{\mathcal{L}}$ a donc la pr\'esentation suivante $\hat{\mathcal{L}} = \{ \hat{X} , \hat{f}_{1} \hat{X} , \ldots , \hat{f}_{p} \hat{X} \} $ avec $[  \hat{X} , \hat{f}_{i}\hat{X}] = \mu \hat{f}_{i} \hat{X}$, $\mu = q r - p s$ et $[  \hat{f_{i}}\hat{X} , \hat{f}_{j}\hat{X}] = 0$. Rappelons que $\hat{N}$ est du type $\hat{a}(x_{1}^{p} x_{2}^{q}) (\lambda_{1} x_{1} \frac{\partial}{\partial x_{1}} + \lambda_{2} x_{2} \frac{\partial}{\partial x_{2}})$. Comme on l'a vu en $(6)$:
$$(r \lambda_{1} + s \lambda_{2}) \hat{\varphi}_{i} + (p \lambda_{1} + q \lambda_{2}) t \hat{\varphi}^{'} \equiv 0.$$

Si $p \lambda_{1} + q \lambda_{2} \neq 0$ alors chaque $\hat{\varphi}_{i}$ est un mon\^ome, et donc $\hat{f}_{i}$ aussi: $\hat{f}_{i} = f_{i} = x_{1}^{r} x_{2}^{s} (x_{1}^{p} x_{2}^{q})^{k_{i}}$. Soient $a$ une r\'ealisation de $\hat{a}$ et $X = (q x_{1} \frac{\partial}{\partial x_{1}} - p x_{2} \frac{\partial}{\partial x_{2}}) + a (x_{1}^{p} x_{2}^{q}) (\lambda_{1} x_{1} \frac{\partial}{\partial x_{1}} + \lambda_{2} x_{2} \frac{\partial}{\partial x_{2}})$. L'alg\`ebre
 $\mathcal{L} = \langle X , f_{1} X , \ldots , f_{p} X  \rangle$ est alors, \`a conjugaison $\mathcal{C}^{\infty}$ pr\`es, une r\'ealisation $\mathcal{C}^{\infty}$ de $\hat{\mathcal{L}}$ telle que $T_{\underline{0}} : \mathcal{L} \rightarrow \hat{\mathcal{L}}$ est un isomorphisme. \\

 Si $p \lambda_{1} + q \lambda_{2} = 0$, alors $\hat{X}$ est de la forme $\hat{X} = \hat{b}(x_{1}^{p} x_{2}^{q})(q x_{1} \frac{\partial}{\partial x_{1}} - p x_{2} \frac{\partial}{\partial x_{2}})$; et puisque $\mu = q r - p s \neq 0$ on v\'erifie qu'en fait $\hat{b}$ est constant.
 Ainsi $\hat{\mathcal{L}} = \langle X = q x_{1} \frac{\partial}{\partial x_{1}} - p x_{2} \frac{\partial}{\partial x_{2}} , \hat{\varphi}_{1}(x_{1}^{p} x_{2}^{q}) X , \ldots , \hat{\varphi}_{p}(x_{1}^{p} x_{2}^{q}) X \rangle$, o\`u $\hat{\varphi}_{i}$ appartient \`a $\hat{\mathcal{E}}_{1}$. Ici aussi, en r\'ealisant les $\hat{\varphi}_{i}$, on obtient une r\'ealisation $\mathcal{C}^{\infty}$ de $\mathcal{L}$ telle que $T_{\underline{0}} : \mathcal{L} \rightarrow \hat{\mathcal{L}}$ soit un isomorphisme. \\
 Lorsque $\hat{X}$ est elliptique on consid\`ere le complexifi\'e $\hat{\mathcal{L}}^{\mathbb{C}}$ de l'alg\`ebre de Lie $\hat{\mathcal{L}}$. Le diff\'eomorphisme $(x_{1} , x_{2}) \mapsto \Phi (x_{1} , x_{2}) = ( x_{1} + i x_{2} , i x_{1} + x_{2})$ conjugue $\hat{X}$ \`a $\hat{Y} = X_{1} + \hat{N}$ o\`u $X_{1} = i \beta (x_{1}\frac{\partial}{\partial x_{1}} - x_{2}\frac{\partial}{\partial x_{2}})$ et $\hat{N}$ est nilpotent ($J^{1}\hat{N}$ = 0). En suivant la preuve de ce qui pr\'ec\`ede on voit que:
 \begin{itemize}
 \item Soit $\hat{X} = \hat{b} (x_{1}^{2} + x_{2}^{2}) X$, o\`u $X = \beta ( - x_{2}\frac{\partial}{\partial x_{1}} + x_{1}\frac{\partial}{\partial x_{2}})$ et $\hat{\mathcal{L}} = \langle \hat{X} , \hat{\varphi}_{1}(x_{1}^{2} + x_{2}^{2}) , \ldots , \hat{\varphi}_{p}(x_{1}^{2} + x_{2}^{2}) \rangle$, o\`u $\hat{b}$ et les $\hat{\varphi}_{i}$ appartiennent \`a $\hat{\mathcal{E}}_{1}$.\\
 \item Soit les vecteurs propres de $\hat{\Phi}_{\ast}\hat{X}^{\mathbb{C}}$ sont des polyn\^omes, $\hat{X}^{\mathbb{C}}$ \'etant le complexifi\'e de $\hat{X}$. Ce qui permet d'en  d\'eduire, \`a conjugaison formelle pr\`es, que $\hat{\mathcal{L}} = \langle \hat{X} , f_{1} \hat{X} , \ldots , f_{p}\hat{X}_{p} \rangle$, o\`u les $f_{i}$ sont des polyn\^omes.
 \end{itemize}
 Selon le cas, en r\'ealisant $\hat{b}$ et les $\hat{\varphi}_{i}$ ou $\hat{X}$, on obtient encore une r\'ealisation $\mathcal{L}$ de $\hat{\mathcal{L}}$ telle que $T_{\underline{0}} : \mathcal{L} \rightarrow \hat{\mathcal{L}}$ est un isomorphisme.\\

{\bf 1.2.3.3 Les cas non semi-simples}

Lorsque $\hat{X}$ a son 1-jet nilpotent, disons $x_{1} \frac{\partial}{\partial x_{2}}$, $\hat{X}$ est lui m\^eme nilpotent. L'exemple typique d'alg\`ebre saturable pr\'esentant cette configuration est $\hat{\mathcal{L}} = \langle \hat{X} = x_{1} \frac{\partial}{\partial x_{2}} , \hat{f}_{1}(x_{1}) \hat{X}, \ldots , \hat{f}_{p}(x_{1}) \hat{X} \rangle$, o\`u les $\hat{f}_{i} \in \hat{\mathcal{E}}_{1}$.\\

Soit $\hat{\mathcal{L}} = \langle \hat{X}  , \hat{f}_{1} \hat{X}, \ldots , \hat{f}_{p} \hat{X} \rangle$ une sous alg\`ebre de Lie saturable de $\hat{\mathcal{X}}_{2}$ avec $J^{1} \hat{X} =  x_{1} \frac{\partial}{\partial x_{2}}$. Comme $\hat{X}$ est nilpotent en suivant la preuve du Lemme 12 on voit que les $\hat{X} \hat{f}_{i}$ sont nuls et que $\hat{\mathcal{L}}$ est ab\'elienne. En particulier si $\dim \hat{\mathcal{L}} \geq 2$, les $\hat{f}_{i}$ sont non constants; et $\hat{X}$ a ainsi une int\'egrale premi\`ere non constante. La Proposition 9 traite ce cas. \\

Dans le cas contraire, \`a diff\'eomorphisme formel  pr\`es,  $\hat{X} = (\lambda x_{1} + x_{2})\frac{\partial}{\partial x_{1}} + \lambda x_{2} \frac{\partial}{\partial x_{2}} $, $\lambda \neq 0$.  Notons qu'on peut supposer que $\lambda = 1$, ce que l'on fera. La Jordanisation du champ $\hat{X}$ est r\'eelle; on en d\'eduit que celle de $\hat{X}_{|E}$ est \'egalement r\'eelle. Consid\'erons une suite $\hat{f}_{1} , \ldots , \hat{f}_{m}$ d'\'el\'ements de $E$ telle que $\hat{X}(\hat{f}_{1}) = \mu \hat{f}_{1}$ et $\hat{X}(\hat{f}_{i}) = \mu \hat{f}_{i} + \hat{f}_{i - 1}$ pour $i = 2 , \ldots , m$. Posons $\hat{f}_{i} = \sum_{j \geq k_{i}} A_{j}^{i}$, o\`u $A_{j}^{i}$ est homog\`ene de degr\'e $j$ et $A_{k_{i}}^{i}$ est non nul. La condition $\hat{X}(\hat{f}_{1}) = \mu \hat{f}_{1}$ implique que:

\begin{equation} (\mu - j) A^{1}_{j} = x_{2} \frac{\partial A^{1}_{j}}{\partial x_{1}} \ \forall j \in \mathbb{N}.
\end{equation}
Si $\mu - j$ est non nul alors $A^{1}_{j} \equiv 0$. En effet si on \'ecrit $A^{1}_{j} = \sum_{k = l}^{j} \alpha_{k}^{j} x_{1}^{k} x_{2}^{j - k}$ en reportant le dans l'\'equation pr\'ec\'edente on voit que n\'ecessairement $\alpha_{l}^{j} = 0$. On en d\'eduit que $\mu = k_{1}$ et $\hat{f}_{1} = \alpha x_{2}^{k_{1}}$. Supposons qu'on ait montr\'e que $\hat{f}_{i - 1}$ est un polyn\^ome. L'\'equation $\hat{X}(\hat{f}_{i}) = \mu \hat{f}_{i} + \hat{f}_{i - 1}$ implique que:

$$ (\mu - j) A^{i}_{j} = x_{2} \frac{\partial A^{i}_{j}}{\partial x_{1}} \ \forall j  > d^{\circ} \hat{f}_{i - 1}. $$

Cette \'equation est du m\^eme type que $(9)$; on en d\'eduit que tous les $A^{i}_{j}$ sont nuls lorsque $j  > d^{\circ} \hat{f}_{i - 1}$ \`a l'exception  peut \^etre d'un seul. Ce qui implique que $\hat{f}_{i}$ est un polyn\^ome. \\

En consid\'erant une base dans laquelle la matrice de $\hat{X}_{|E}$ est sous forme de Jordan on d\'eduit que $\hat{\mathcal{L}}$ est conjugu\'ee \`a $\langle X = ( x_{1} + x_{2})\frac{\partial}{\partial x_{1}} +  x_{2} \frac{\partial}{\partial x_{2}} , f_{1} X , \ldots , f_{p} X \rangle$, o\`u les $f_{i}$ sont des polyn\^omes. En r\'esum\'e nous obtenons le:

\begin{theorem} Soit $\hat{\mathcal{L}} \subset \hat{\mathcal{X}}_{2}$ une sous alg\`ebre de Lie de dimension finie, saturable et de rang ponctuel 1. Il existe une r\'ealisation $\mathcal{C}^{\infty}$ not\'ee $\mathcal{L}$ de $\hat{\mathcal{L}}$ telle que $T_{\underline{0}} : \mathcal{L} \rightarrow \hat{\mathcal{L}}$ soit un isomorphisme d'alg\`ebre de Lie.
\end{theorem}

\subsection{Dimension deux: alg\`ebres ab\'eliennes de rang deux}

La Proposition 7 montre que les alg\`ebres ab\'eliennes jouent un r\^ole sp\'ecial dans notre contexte. Toutefois cette proposition ne poss\`ede pas de g\'en\'eralisa-tion lorsque le rang ponctuel g\'en\'erique est plus grand que 1. Par exemple l'alg\`ebre engendr\'ee par les trois champs $X_{0} = x_{1} x_{2} \frac{\partial}{\partial x_{3}}$, $X_{1} = x_{2} x_{3} \frac{\partial}{\partial x_{4}}$ et $X_{2} = x_{1} x_{2}^{2} \frac{\partial}{\partial x_{4}}$, $[X_{0} , X_{1}] = X_{2}$, est une sous alg\`ebre  de rang 2 de $\hat{\mathcal{X}}_{4}$; elle est toutefois nilpotente. En fait on a la:

\begin{proposition} Soit $\hat{\mathcal{L}} \subset \hat{\mathcal{X}}_{n}$ une sous alg\`ebre de Lie de dimension finie. On suppose que $J^{1} \hat{\mathcal{L}} = \{ 0 \}$, i.e. $\hat{\mathcal{L}} \subset \mathcal{M}^{2} \hat{\mathcal{X}}_{n}$. Alors $\hat{\mathcal{L}}$ est nilpotente.
\end{proposition}

\begin{proof}
Il suffit d'\'etablir que les applications $ad_{\hat{X}} : \hat{\mathcal{L}} \rightarrow \hat{\mathcal{L}}$ d\'efinies par $ad_{\hat{X}}(\hat{Y}) = [\hat{X} , \hat{Y}]$ et $\hat{X} \in \hat{\mathcal{L}}$, sont toutes nilpotentes d'apr\`es \cite{Bourbaki}. On peut, quitte \`a complexifier, supposer que $\hat{\mathcal{L}}$ est d\'efinie sur $\mathbb{C}$, i.e. $\hat{\mathcal{L}} \subset \hat{\mathcal{X}}_{n}(\mathbb{C}^{n}_{,0})$. Soient $\hat{X}$ et $\hat{Y}$ deux \'el\'ements de $\hat{\mathcal{L}}$ tels que $\hat{Y}$ est propre pour $ad_{\hat{X}}$: $ad_{\hat{X}}(\hat{Y}) = \mu \hat{Y}$. Comme $J^{1} \hat{X} = 0$ l'ordre du premier jet non nul de $[\hat{X} , \hat{Y}]$ est strictement sup\'erieur \`a celui de $\hat{Y}$; on en d\'eduit que $\mu$ est nulle et par suite $ad_{\hat{X}}$ est nilpotente. 
\end{proof}

On appelle quotient formel tout \'el\'ement du corps des fractions $\hat{\mathcal{M}}_{n}$ de l'anneau des s\'eries formelles de $\hat{\mathcal{E}}_{n}$. Un \'el\'ement de $\hat{\mathcal{M}}_{n}$ s'\'ecrit $\frac{\hat{f}}{\hat{g}}$, o\`u $\hat{f}$ et $\hat{g}$ sont des \'el\'ements de $\hat{\mathcal{E}}_{n}$ sans facteur commun. Pour les alg\`ebres commutatives de champs formels, on obtient en dimension deux d'espace:

\begin{lemma} Soit $\hat{\mathcal{L}} \subset \hat{\mathcal{X}}_{2}$ une alg\`ebre de Lie ab\'elienne de champs formels de rang $2$. Alors $\dim \hat{\mathcal{L}} = 2$.
\end{lemma}

{\bf Preuve} Soient $\hat{X}$ et $\hat{Y}$ deux \'el\'ements de $\hat{\mathcal{L}}$ tels que $det ( \hat{X} , \hat{Y} ) \neq 0$. Si $\hat{Z}$ appartient \`a $\hat{\mathcal{L}}$ il existe des quotients  formels $\hat{A}$ et $\hat{B}$ tels que $\hat{Z} = \hat{A} \hat{X} + \hat{B} \hat{Y}$ (alg\`ebre lin\'eaire sur le corps des s\'eries formelles). Comme $[\hat{X} , \hat{Y}] = [\hat{X} , \hat{Z}] = [\hat{Y} , \hat{Z}] = 0$  on a $\hat{X} ( \hat{A} ) = \hat{X} ( \hat{B} ) = \hat{Y} ( \hat{A} ) = \hat{Y} ( \hat{B} ) = 0$. Par suite on en d\'eduit que $\frac{\partial \hat{A}}{\partial x_{i}} = \frac{\partial \hat{B}}{\partial x_{i}} = 0$, pour $i = 1 , 2$. Ce qui implique que $\hat{A}$ et $\hat{B}$ sont constantes.

\begin{examples} Dans chacun des cas suivants l'alg\`ebre $\hat{\mathcal{L}}$ est ab\'elienne de rang $2$.
\begin{enumerate}
\item Variables s\'epar\'ees: $\hat{\mathcal{L}} = \langle \hat{f}_{1} ( x_{1} ) \frac{\partial}{\partial x_{1}} , \hat{f}_{1} ( x_{2} ) \frac{\partial}{\partial x_{2}} \rangle$, o\`u $\hat{f}_{i}$ appartient \`a $\hat{\mathcal{E}}_{1}$.
\item Lin\'eaires diagonales:  $\hat{\mathcal{L}} = \langle \lambda_{1} x_{1}  \frac{\partial}{\partial x_{1}} + \lambda_{2} x_{2}  \frac{\partial}{\partial x_{2}} , \mu_{1} x_{1}  \frac{\partial}{\partial x_{1}} + \mu_{2} x_{2}  \frac{\partial}{\partial x_{2}} \rangle$, avec $ \lambda_{1} \mu_{2} - \lambda_{2} \mu_{1} \neq 0$.
\item R\'esonnantes: $\hat{\mathcal{L}} = \langle q x_{1} \frac{\partial}{\partial x_{1}} - p x_{2} \frac{\partial}{\partial x_{2}} , \hat{a} ( x_{1}^{p} x_{2}^{q} ) ( \lambda_{1} x_{1}  \frac{\partial}{\partial x_{1}} + \lambda_{2} x_{2}  \frac{\partial}{\partial x_{2}} ) \rangle$, o\`u $\lambda_{i} \in \mathbb{R}$, $p , q \in \mathbb{N}$ et $p \lambda_{1} + q \lambda_{2} \neq 0$.
\end{enumerate}

On dispose, dans chacun de ces cas, de r\'ealisation $\mathcal{C}^{\infty}$  $\mathcal{L}$ de $\hat{\mathcal{L}}$ telle que $T_{\underline{0}} : \mathcal{L} \rightarrow \hat{\mathcal{L}}$ soit un isomorphisme.
\end{examples}
Si $\hat{X}$ et $\hat{Y}$ sont tels que $det ( \hat{X} , \hat{Y} ) \neq 0$ il existe deux 1-formes diff\'erentielles $\hat{\alpha}$ et $\hat{\beta}$, \`a coefficients dans $\hat{\mathcal{M}}_{2}$ telles que $i_{\hat{X}} \hat{\alpha} = i_{\hat{Y}} \hat{\beta} = 1$ et $i_{\hat{X}} \hat{\beta} = i_{\hat{Y}} \hat{\alpha} = 0$. La commutation de $\hat{X}$ et $\hat{Y}$ implique que les 1-formes $\hat{\alpha}$ et $\hat{\beta}$ sont ferm\'ees.

\subsubsection{Formes normales des formes ferm\'ees \`a coefficient dans $\hat{\mathcal{M}}_{n}$}
Consid\'erons un germe de 1-forme m\'eromorphe ferm\'ee $\omega$ \`a l'origine de $\mathbb{C}^{n}$; il s'\'ecrit sous la forme $\omega = \frac{\Theta}{f}$, avec $\Theta \in \Omega( \mathbb{C}^{n} )$  germes de 1-forme holomorphe et $f \in \mathcal{O} ( \mathbb{C}^{n} )$,  $f = f_{1}^{n_{1} + 1} \ldots f_{p}^{n_{p} + 1}$, les $f_{i}$ \'etant irr\'eductibles et aucun des $f_{i}$ ne divisant $\Theta$. D. Cerveau et J.-F. Mattei \cite{Cerveau2} \'etablissent la d\'ecomposition, en "\'el\'ements simples", suivante de $\omega$:

$$ \omega = \sum_{i = 1}^{p} \lambda_{i} \frac{d f_{i}}{f_{i}} + d ( \frac{ H}{f_{1}^{n_{1}} \ldots f_{p}^{n_{p} }} ) $$
o\`u $\lambda_{i} \in \mathbb{C}$ (r\'esidu de $\omega$ le long de $f_{i}$) et $H \in \mathcal{O} ( \mathbb{C}^{n} )$. Dans le cas r\'eduit o\`u $\omega$ est \`a p\^oles simples, $ \omega = \sum_{i = 1}^{p} \lambda_{i} \frac{d f_{i}}{f_{i}} + d H$, $\omega$ est dite logarithmique. La d\'ecomposition en \'el\'ements simples s'\'etend aux formes m\'eromorphes formelles ferm\'ees, c'est \`a dire \`a coefficients dans le corps des fractions $\hat{\mathcal{M}}_{n} ( \mathbb{C} )$ de $\hat{\mathcal{O}} ( \mathbb{C}^{n} )$. Pour un tel $\hat{\omega}$ on a de mani\`ere analogue:

$$ \hat{\omega} = \sum_{i = 1}^{p} \lambda_{i} \frac{d \hat{f}_{i}}{\hat{f}_{i}} + d ( \frac{ \hat{H}}{\hat{f}_{1}^{n_{1}} \ldots \hat{f}_{p}^{n_{p} }} ). $$

Consid\'erons \`a pr\'esent une 1-forme m\'eromorphe formelle ferm\'ee $\hat{\omega} = \frac{\hat{\Theta}}{\hat{f}}$ o\`u $\hat{\Theta} = \sum_{i = 1}^{n} \hat{a}_{i} d x_{i} \in \hat{\Omega}^{1}_{n}$ et les $\hat{a}_{i} , \hat{f} \in \hat{\mathcal{E}}_{n}$. Notons $\hat{f}_{\mathbb{C}}$  le complexifi\'e de $\hat{f}$. La d\'ecomposition en facteurs irr\'eductibles de $\hat{f}_{\mathbb{C}}$ est du type
$\hat{f}_{\mathbb{C}} = \hat{f}_{1}^{n_{1} + 1} \ldots \hat{f}_{q}^{n_{q} + 1} (\hat{g}_{1}\hat{h}_{1} )^{m_{1} + 1} \ldots (\hat{g}_{l}\hat{h}_{l} )^{m_{l} + 1}$, o\`u les $\hat{f}_{1}, \ldots , \hat{f}_{q}$ sont r\'eels, i.e. des complexifi\'es d'\'el\'ements de $\hat{\mathcal{E}}_{n}$, et les $\hat{g}_{j}$ et $\hat{h}_{j}$ sont complexes conjugu\'es; ce qui revient \`a dire que $\hat{g}_{j} \hat{h}_{j}$ est le complexifi\'e d'un \'el\'ement de  $\hat{\mathcal{E}}_{n}$  du type $\hat{P}_{j}^{2} + \hat{Q}_{j}^{2} = ( \hat{P}_{j} + i \hat{Q}_{j} ) ( \hat{P}_{j} - i \hat{Q}_{j} )$.

Si $\hat{\omega}^{\mathbb{C}}$ est le complexifi\'e de $\hat{\omega}$ on a:

$$ \hat{\omega}^{\mathbb{C}} = \sum_{i = 1}^{p} \lambda_{i} \frac{d \hat{f}_{i}}{\hat{f}_{i}} + \sum_{j = 1}^{l} \mu_{j} \frac{d \hat{g}_{j}}{\hat{g}_{j}} + \sum_{j = 1}^{l} \bar{\mu}_{j} \frac{d \hat{h}_{j}}{\hat{h}_{j}} + d \left( \frac{\hat{K}}{\hat{f}_{1}^{n_{1}} \ldots \hat{f}_{p}^{n_{p} } (\hat{g}_{1}\hat{h}_{1})^{m_{1}} \ldots (\hat{g}_{l}\hat{h}_{1})^{m_{l}}} \right) $$
avec $\lambda_{i} \in \mathbb{R}$, $\mu_{j} \in \mathbb{C}$, $\hat{K} \in \hat{\mathcal{O}} (\mathbb{C}^{n}_{, 0} )$ r\'eel. La forme $\hat{\omega}^{\mathbb{C}}$ est invariante sous l'action de l'automorphisme de corps $z \mapsto \bar{z}$. De sorte que $\hat{\omega}$ s'\'ecrit finalement sous la forme $(\ast \ast)$:

$$\hat{\omega} = \sum_{i = 1}^{p} \lambda_{i} \frac{d \hat{f}_{i}}{\hat{f}_{i}} + \sum_{j = 1}^{l} (a_{j} \frac{d ( \hat{P}_{j}^{2} + \hat{Q}_{j}^{2} ) }{( \hat{P}_{j}^{2} + \hat{Q}_{j}^{2} )} + b_{j} \frac{\hat{P}_{j} d \hat{Q}_{j} - \hat{Q}_{j} d \hat{P}_{j})}{( \hat{P}_{j}^{2} + \hat{Q}_{j}^{2} }) + $$ $$d \left( \frac{\hat{H}}{\hat{f}_{1}^{n_{1}} \ldots \hat{f}_{p}^{n_{p} } ( \hat{P}_{1}^{2} + \hat{Q}_{1}^{2} )^{m_{1}} \ldots ( \hat{P}_{l}^{2} + \hat{Q}_{l}^{2} )^{m_{l}}}\right) $$
avec $\lambda_{i},  a_{j} , b_{j} \in \mathbb{R}$, $\hat{f}_{i} , \hat{P}_{j} , \hat{Q}_{j}$ et $\hat{H}$ dans $\hat{\mathcal{E}}_{n}$.
On dit que $\hat{\omega}$ est logarithmique s'il est \`a p\^oles simples; c'est \`a dire les $n_{i}$ et $m_{j}$ sont tous nuls.

\subsubsection{Classification}

Revenons \`a une sous alg\`ebre commutative $\hat{\mathcal{L}} = \langle \hat{X} , \hat{Y} , [ \hat{X} , \hat{Y} ] = 0 \rangle $ de $\hat{\mathcal{X}}_{2}$. Supposons que l'un des \'el\'ements de $\hat{\mathcal{L}}$, disons $\hat{X}$, soit non singulier. Il existe un diff\'eomorphisme formel $\hat{\Phi}$ tel que
$\hat{\Phi}_{\ast}\hat{X} = \frac{\partial}{\partial x_{1}}$.  Un calcul \'el\'ementaire montre que $\hat{\Phi}_{\ast}\hat{Y}$ s'\'ecrit sous la forme $ \hat{a} ( x_{2} ) \frac{\partial}{\partial x_{1}} +  \hat{b}(x_{2} ) \frac{\partial}{\partial x_{2}}$, o\`u $\hat{a}$ et $\hat{b}$ appartiennent \`a $\hat{\mathcal{E}}_{1}$. Soient  $a$, $b \in \mathcal{E}_{1}$ et  $\Phi \in \mathrm{Diff}_{2}$ des r\'ealisations de $\hat{a}$, $\hat{b}$ et $\hat{\Phi}$ respectivement. En posant $X = \Phi^{- 1}_{\ast}  \frac{\partial}{\partial x_{1}}$ et $Y = \Phi^{- 1}_{\ast} ( a  \frac{\partial}{\partial x_{1}} +  b \frac{\partial}{\partial x_{2}} )$ on obtient une r\'ealisation $\mathcal{L} = \langle X , Y \rangle$ de $\hat{\mathcal{L}}$ telle que $T_{\underline{0}}: \mathcal{L} \rightarrow \hat{\mathcal{L}}$ est un isomorphisme d'alg\`ebre.

Dans la suite on supposera que tous les \'el\'ements de $\hat{\mathcal{L}} = \langle \hat{X} , \hat{Y} , [ \hat{X} , \hat{Y} ] = 0 \rangle $ sont singuliers, i.e. $\hat{X} ( 0 ) = \hat{Y} ( 0 ) = 0$. Comme pr\'ec\'edemment on note $\mathcal{L}^{1}$ l'alg\`ebre des 1-jets des \'el\'ements de $\hat{\mathcal{L}}$, $\hat{\alpha}$ et $\hat{\beta}$ les 1-formes ferm\'ees \`a coefficients dans $\hat{\mathcal{M}}_{2}$ d\'efinies par $i_{\hat{X}} \hat{\alpha} = i_{\hat{Y}} \hat{\beta} = 1$ et $i_{\hat{X}} \hat{\beta} = i_{\hat{Y}} \hat{\alpha} = 0$. On note $\hat{\mathcal{L}}^{\star}$ l'espace vectoriel $\hat{\mathcal{L}}^{\star} = \hat{\alpha} \mathbb{R} + \hat{\beta} \mathbb{R}$. On dira que $\hat{\mathcal{L}}$ est logarithmique si l'\'el\'ement g\'en\'erique de $\hat{\mathcal{L}}^{\star}$ est logarithmique et semi-logarithmique si $\hat{\mathcal{L}}^{\star}$ contient une 1-forme logarithmique non triviale.

Les alg\`ebres lin\'eaires diagonales $ \langle x_{1} \frac{\partial}{\partial x_{1}} , x_{2} \frac{\partial}{\partial x_{2}} \rangle$ sont logarithmiques tandis que les alg\`ebres r\'esonnantes $\langle q x_{1} \frac{\partial}{\partial x_{1}} - p x_{2} \frac{\partial}{\partial x_{2}} , \hat{q} ( x_{1}^{p} x_{2}^{q} )(\lambda_{1} x_{1} \frac{\partial}{\partial x_{1}} + \lambda_{2} x_{2} \frac{\partial}{\partial x_{2}} ) \rangle$ sont semi-logarithmiques.

\begin{proposition} Soit $\hat{\mathcal{L}} \subset \hat{\mathcal{X}}_{2}$ une alg\`ebre ab\'elienne de rang $2$ telle que $\hat{\mathcal{L}} ( 0 ) = 0$. Si $\hat{\mathcal{L}}$ est semi-logarithmique, alors $\mathcal{L}^{1}$ est non nul. De plus $\mathcal{L}^{1}$ contient un \'el\'ement non nilpotent.
\end{proposition}

\begin{proof}
On conserve les notations pr\'ec\'edentes; supposons que $\hat{\alpha}$ est logarithmique et $i_{\hat{X}} \hat{\alpha} = 1$. Il est plus commode de travailler avec les complexifi\'es $\hat{\alpha}_{\mathbb{C}}$ et $\hat{X}_{\mathbb{C}}$ de $\hat{\alpha}$ et $\hat{X}$ respectivement. Puisque $\hat{\alpha}$ est logarithmique on a $\hat{\alpha}_{\mathbb{C}} = \sum \lambda_{i} \frac{d \hat{f}_{i}}{f_{i}} + d \hat{H}$, $\hat{f}_{i} , \hat{H} \in \hat{\mathcal{O}} ( \mathbb{C}^{2} )$. Comme les $\hat{f}_{i}$ sont irr\'eductibles la condition $i_{\hat{X}} \hat{\alpha} = 1$ implique que $\hat{X}_{\mathbb{C}} ( \hat{f}_{i} ) = \mu_{i} \hat{f}_{i}$, $\mu_{i} \in \hat{\mathcal{O}} ( \mathbb{C}^{2} )$. On a alors $\sum \mu_{i} \hat{f}_{i} + \hat{X}_{\mathbb{C}} ( \hat{H}) =1$ et comme $\hat{X}_{\mathbb{C}} ( 0 ) = 0$ l'un des $\mu_{i} ( 0 )$, disons $\mu_{1} ( 0 )$, est non nul. La condition $\hat{X}_{\mathbb{C}} ( \hat{f}_{1} ) = \mu_{1} \hat{f}_{1}$ implique la non nullit\'e de $J^{1} \hat{X}_{\mathbb{C}}$ et donc de $J^{1} \hat{X}$. Remarquons que le premier jet non nul de $\hat{f}_{1}$ est propre pour la d\'erivation $J^{1} \hat{X}$ avec pour valeur propre $\mu_{1} ( 0 )$ diff\'erente de $0$. Par suite $J^{1} \hat{X}$ est non nilpotent. 
\end{proof}

A conjugaison lin\'eaire et constante multiplicative pr\`es les 1-jets des champs non nilpotents qui vont entrer en jeu sont les suivants:

\begin{enumerate}
\item $\lambda_{1} x_{1} \frac{\partial}{\partial x_{1}} + \lambda_{2} x_{2} \frac{\partial}{\partial x_{2}}$ sans r\'esonnance, i.e. $i_{1} \lambda_{1} + i_{2} \lambda_{2} \neq\lambda_{j}$, $j = 1 , 2$, si $i_{1}$ et  $i_{2}$ sont des entiers tels que $i_{1} + i_{j} \geq 2$.
\item $ x_{1} \frac{\partial}{\partial x_{1}} + n x_{2} \frac{\partial}{\partial x_{2}}$, $n \geq 2$ (Poincar\'e-Dulac).
\item $q x_{1} \frac{\partial}{\partial x_{1}} - p x_{2} \frac{\partial}{\partial x_{2}}$, $p , q \in \mathbb{N}^{\ast}$, $\langle p , q \rangle = 1$ (hyperbolique r\'esonnante de type $( p , q)$.
\item $ x_{1} \frac{\partial}{\partial x_{1}} + (x_{1} + x_{2}) \frac{\partial}{\partial x_{2}}$.
\item $(\alpha x_{1} + \beta x_{2})\frac{\partial}{\partial x_{1}} + (- \beta x_{1} + \alpha x_{2} ) \frac{\partial}{\partial x_{2}}$, $\alpha \neq 0$.
\item $ x_{2} \frac{\partial}{\partial x_{1}} -  x_{1} \frac{\partial}{\partial x_{2}}$ (elliptique).
\item $ x_{1} \frac{\partial}{\partial x_{1}}$ (noeud-col).
\end{enumerate}

Les 1-jets de type 1, 4, et 5 sont 1-d\'eterminants, i.e. si $\hat{X}$ a son 1-jet   de type 1 , 4 ou 5 alors $\hat{X}$ est formellement conjugu\'e \`a sa partie lin\'eaire $X_{1} = J^{1} \hat{X}$. De plus un champ $\hat{Y}$ commutant avec  $X_{1} $ est lin\'eaire. En utilisant la liste pr\'ec\'edente et la d\'ecomposition de Jordan des champs de vecteurs formels en parties semi-simple et nilpotente  \cite{Bourbaki}, \cite{Cerveau2}  on \'etablit sans peine la:

\begin{proposition} Soit $\hat{\mathcal{L}} \subset \hat{\mathcal{X}}_{2}$ une sous alg\`ebre commutative de rang 2, $\hat{\mathcal{L}} (0) = \{ 0 \}$. Si $\hat{\mathcal{L}}$ est semi-logarithmique alors, \`a conjugaison par un \'el\'ement de $\widehat{\mathrm{Diff}}(\mathbb{R}^{2}_{0})$ pr\`es, $\hat{\mathcal{L}}$ appartient \`a la liste suivante:
\begin{enumerate}
\item $\langle  x_{1} \frac{\partial}{\partial x_{1}} ,  x_{2} \frac{\partial}{\partial x_{2}} \rangle$
\item $\langle x_{1} \frac{\partial}{\partial x_{1}}  + n x_{2} \frac{\partial}{\partial x_{2}} , x_{1}^{n} \frac{\partial}{\partial x_{2}} \rangle$, o\`u $n \geq 2$.
\item $\langle q x_{1} \frac{\partial}{\partial x_{1}} - p x_{2} \frac{\partial}{\partial x_{2}} , \hat{a} (x_{1}^{p}x_{2}^{q}) (\lambda_{1} x_{1} \frac{\partial}{\partial x_{1}} + \lambda_{2} x_{2} \frac{\partial}{\partial x_{2}}) \rangle$, $\hat{a} \in \mathcal{M} .\hat{\mathcal{E}}_{1}$, $\lambda_{1} , \lambda_{2} \in \mathbb{R}$.
\item $\langle x_{1} \frac{\partial}{\partial x_{1}} +  x_{2} \frac{\partial}{\partial x_{2}} , x_{1}\frac{\partial}{\partial x_{2}} \rangle$.
\item $\langle  x_{1} \frac{\partial}{\partial x_{1}} +  x_{2} \frac{\partial}{\partial x_{2}} , x_{2} \frac{\partial}{\partial x_{1}} -  x_{1} \frac{\partial}{\partial x_{2}} \rangle$
\item $ x_{2} \frac{\partial}{\partial x_{1}} -  x_{1} \frac{\partial}{\partial x_{2}} , \alpha \hat{a} ( x_{1} \frac{\partial}{\partial x_{1}} +  x_{2} \frac{\partial}{\partial x_{2}}) + \beta \hat{a} (x_{2} \frac{\partial}{\partial x_{1}} -  x_{1} \frac{\partial}{\partial x_{2}}) \rangle$, o\`u $\hat{a}  \in \mathcal{M} .\hat{\mathcal{E}}_{1}$ et  $\alpha , \beta \in \mathbb{R}$.
\item $\langle x_{1} \frac{\partial}{\partial x_{1}} , \hat{a}(x_{2}) \frac{\partial}{\partial x_{2}} \rangle$, o\`u $\hat{a} \in \mathcal{M}^{2} .\hat{\mathcal{E}}_{1}$
\end{enumerate}
\end{proposition}

Remarquons que les alg\`ebres de type 1 et 5 sont logarithmiques et les autres sont semi-logarithmiques. On constate ainsi que logarithmique implique lin\'earisable. Les alg\`ebres de type 4 sont de type 2 lorsque $n = 1$, mais sont lin\'earisables. On peut en fait donner des formes normales plus pr\'ecises pour les $\hat{a}$, $\hat{l}_{1}$ et $\hat{l}_{2}$, mais ceci ne sera pas utile ($\hat{a} (t) = \frac{t^{s + 1}}{1 - \lambda t^{s}}$, $s \in \mathbb{N}^{\ast}$ et $\lambda \in \mathbb{R}$). On d\'eduit de tout cela le:

\begin{theorem} Soit $\hat{\mathcal{L}} \subset \hat{\mathcal{X}}_{2}$ une sous alg\`ebre ab\'elienne de rang $2$. Il existe une r\'ealisation $\mathcal{C}^{\infty}$ $\mathcal{L}$ de $\hat{\mathcal{L}}$ telle que $T_{\underline{0}}: \mathcal{L} \rightarrow \hat{\mathcal{L}}$  soit un isomorphisme dans les deux cas suivants:
\begin{enumerate}
\item $\hat{\mathcal{L}}(0) \neq 0$.
\item $\hat{\mathcal{L}}(0) = 0$ et $\hat{\mathcal{L}}$ est semi-logarithmique.
\end{enumerate}
\end{theorem}

\begin{proof}
Soit $\hat{\Phi}$ un \'el\'ement de $\widehat{\mathrm{Diff}}(\mathbb{R}^{2}_{0})$ tel que $\hat{\Phi}_{\ast}\hat{\mathcal{L}} = \hat{\mathcal{L}}_{0}$ soit comme dans la Proposition 18. Si $\hat{\mathcal{L}}_{0}$ est de type 3, 5, 6 ou 7 on choisit une r\'ealisation $\mathcal{C}^{\infty}$, $a \in \mathcal{M}\mathcal{E}_{1}$, de $\hat{a}$. On construit ainsi des alg\`ebres $\mathcal{L}_{0}$, r\'ealisations $\mathcal{C}^{\infty}$ de $\hat{\mathcal{L}}_{0}$. La sous alg\`ebre $\mathcal{L} = \Phi_{\ast} \mathcal{L}_{0}$, o\`u $\Phi$ est une r\'ealisation de $\hat{\Phi}$, satisfait le th\'eor\`eme. 
\end{proof}

\begin{remark} Nous n'avons pas abord\'e ici le cas g\'en\'eral non semi-loga-rithmique pour lequel il n'y a pas \`a notre connaissance de mod\`eles comme dans la Proposition 18.
\end{remark}

\section{Equations implicites}
Rappelons les d\'efinitions et propri\'et\'es des id\'eaux  ferm\'es de $\mathcal{E}_{n}$, afin d'\'enoncer le principal r\'esultat de J. C. Tougeron \cite{Tougeron1} et \cite{Tougeron2}, que nous utiliserons. Nous conservons ses notations. Soient $\Omega$ un ouvert de $\mathbb{R}^{n}$ et $\mathcal{E}(\Omega)$ l'alg\`ebre des fonctions num\'eriques ind\'efiniment d\'erivables sur $\Omega$ muni de sa structure classique d'espace de Frechet.  Etant donn\'es $\underline{a} \ \in \ \Omega$ et $ \varphi \ \in \ \mathcal{E} (\Omega)$ on note $T_{\underline{a}} \varphi$ le d\'eveloppement en s\'erie formelle de $\varphi$ en $a$. Si $I$ est un id\'eal, on note $T_{\underline{a}} I = \{ T_{\underline{a}} \varphi \  / \ \varphi \ \in \ I \} \subset \hat{\mathcal{E}}_{n}$. On dira que $I$ est ferm\'e si c'est un ferm\'e de $\mathcal{E} (\Omega)$ muni de sa structure d'espace de Frechet. Nous avons le r\'esultat suivant, d\^u \`a Whitney, qui carat\'erise les id\'eaux ferm\'es de $\mathcal{E} ( \Omega )$:

\begin{theorem} \cite{Tougeron1}, \cite{Malgrange} Un id\'eal $I$ de $\mathcal{E} (\Omega)$ est ferm\'e si et seulement si pour tout $\varphi$ appartenant \`a $\mathcal{E} (\Omega)$ tel que $T_{\underline{a}} \varphi  \ \in \ T_{\underline{a}} I$ pour tout $\underline{a}$ \'el\'ement de $\Omega$, alors $\varphi$ appartient \`a $I$.
\end{theorem}

H\"ormander a montr\'e que l'id\'eal $\varphi \mathcal{E} (\Omega)$ est ferm\'e lorsque $\varphi$ est un polyn\^ome. Lojasiewicz a montr\'e le m\^eme r\'esultat sous l'hypoth\`ese que $\varphi$ est analytique sur $\Omega$. Plus g\'en\'eralement Malgrange a prouv\'e que:

\begin{theorem} \cite{Malgrange}, \cite{Tougeron1} Soient $\varphi_{1}$, $\ldots$ , $\varphi_{p}$ des fonctions analytiques sur $\Omega$ et $<\varphi_{1} , \ldots , \varphi_{n}>$ l'id\'eal engendr\'e par $\varphi_{1}$, $\ldots$ , $\varphi_{p}$, alors $<\varphi_{1} , \ldots , \varphi_{n}>$ est ferm\'e.
\end{theorem}

La notion d'id\'eaux ferm\'es se g\'en\'eralise \`a ceux de $\mathcal{E}_{n}$ de type fini.

\begin{definition} Soit $I = <\varphi_{1} , \ldots , \varphi_{p}>$ un id\'eal de $\mathcal{E}_{n}$. On dit que $I$ est ferm\'e s'il existe un voisinage $\Omega$ de $\underline{0}$ et des repr\'esentants $\tilde{\varphi}_{1}$, $\ldots$ , $\tilde{\varphi}_{p}$ de $\varphi_{1}, \ldots , \varphi_{p}$ respectivement tels que $< \tilde{\varphi}_{1} , \ldots , \tilde{\varphi}_{p} >$ soit ferm\'e dans $\mathcal{E} (\Omega)$.
\end{definition}

\begin{example} Soit $\varphi \ \in \ \mathcal{E}_{2}$ non plat, i.e. $T_{\underline{0}} \varphi \neq 0$. Alors $\varphi \mathcal{E}_{2}$ est ferm\'e. En effet, d'apr\`es D. Cerveau et R. Mattei \cite{Cerveau2}, il existe un diff\'eomorphisme formel $\hat{\Phi}$ de $\widehat{\mathrm{Diff}} (\mathbb{R}^{2}_{0})$ tel que $(T_{\underline{0}} \varphi) \circ \hat{\Phi} = P$, $P$ \'etant un polyn\^ome. Soit $\Phi$ une r\'ealisation de Borel de $\hat{\Phi}$, $P \circ \Phi^{- 1^{}}$ est une r\'ealisation de Borel de $T_{\underline{0}} \varphi$ et est $\mathcal{C}^{\infty}$ conjugu\'e \`a un polyn\^ome. On en d\'eduit que $(P \circ \Phi^{- 1}) \mathcal{E}_{2}$ est ferm\'e, et par suite $\varphi \mathcal{E}_{2}$.
\end{example}

Voici un exemple d'id\'eal non ferm\'e donn\'e par Tougeron \cite{Tougeron1}:

\begin{example} La fonction $ f^{+} = y^{2} + \exp{( - \frac{1}{x^{2}})}$ n'engendre pas un id\'eal ferm\'e. En effet $\forall \ \underline{a} \ \in \mathbb{R}^{2}$, $T_{\underline{a}} \exp{( - \frac{1}{x^{2}})} \ \in \ T_{\underline{a}} ( f^{+} \mathcal{E} (\mathbb{R}^{2})$, cependant $\exp{( - \frac{1}{x^{2}})}$ n'appartient pas \`a $ f^{+} \mathcal{E} (\mathbb{R}^{2})$.
\end{example}

\begin{definition} Soit $I = <\varphi_{1} , \ldots , \varphi_{p}>$ l'id\'eal de $\mathcal{E} (\Omega)$ engendr\'e par $\varphi_{1}$, $\ldots$ , $\varphi_{p}$. On appelle module des relations de $I$ le $\mathcal{E} (\Omega)$-module $\{ (h_{1} , \ldots , h_{p} ) \ / h_{i} \ \in \ \mathcal{E}(\Omega) , \sum_{i = 1}^{p} h_{i} \varphi_{i} = 0 \}$.
\end{definition}

On a le r\'esultat remarquable suivant, d\^u \`a Malgrange:

\begin{theorem} Soit $I = <\varphi_{1} , \ldots , \varphi_{p}>$ un id\'eal de $\mathcal{E}_{n}$, $\hat{h}_{1}$, $\ldots$ , $\hat{h}_{p}$ des \'el\'ements de $\hat{\mathcal{E}}_{n}$ tels que $ \sum_{i = 1}^{p} \hat{h}_{i} T_{\underline{0}} \varphi_{i} = 0$. Si $I$ est ferm\'e, alors il existe des \'el\'ements $h_{1}, \ldots , h_{p}$ de $\mathcal{E}_{n}$ tels que $\sum_{i = 1}^{p} h_{i}  \varphi_{i} = 0$.
\end{theorem}

Rappelons dans le m\^eme ordre d'id\'ee le c\'el\`ebre r\'esultat d'approximation d'Artin \cite{Artin}:

\begin{theorem} Soit $F$ un germe d'application holomorphe \`a l'origine de $\mathbb{C}^{n} \times \mathbb{C}^{p}$. On consid\`ere l'\'equation implicite:
\begin{equation} F( x , y ) = 0
\end{equation}
o\`u $ x$ appartient \`a  $\mathbb{C}^{n}$ et  $y$ \`a $\mathbb{C}^{p}$. Si $(10)$ poss\`ede une solution formelle $\hat{y} \in \hat{\mathcal{O}}(\mathbb{C}^{n}_{, 0})^{p}$ et $k \in \mathbb{N}$ est fix\'e, alors il existe une solution convergente $y_{k} \in\mathcal{O}(\mathbb{C}^{n}_{, 0})^{p}$ de $(10)$ telle que le jet d'ordre $k$ de $y_{k}$ v\'erifie $J^{k} y_{k}  = J^{k} \hat{y}$.
\end{theorem}

Cet \'enonc\'e poss\`ede une version analytique r\'eelle qui a \'et\'e compl\'et\'ee comme suit par J. C. Tougeron \cite{Tougeron3}:

\begin{theorem} Soient $F$ un germe d'application analytique \`a l'origine de $\mathbb{R}^{n} \times \mathbb{R}^{p}$ et $\hat{y} \in \hat{\mathcal{E}}_{n}^{p}$ une solution formelle de l'\'equation implicite $(10)$. Il existe un germe d'application $\mathcal{C}^{\infty}$, $y \in \mathcal{E}_{n}^{p}$ v\'erifiant:
$ F( x , y ) = 0 $ et $T_{\underline{0}} y = \hat{y}$.
\end{theorem}

On consid\`ere maintenant la donn\'ee d'une \'equation implicite "compl\`etement formelle"
\begin{equation}
\hat{F}( x , y ) = 0
\end{equation}
avec $\hat{F} \in \hat{\mathcal{E}}_{n + p}^{l}$. On suppose qu'il existe une solution formelle $\hat{y} \in \hat{\mathcal{E}}_{n}^{p}$ de ${(11)}$ : $\hat{F}(x, \hat{y}) = 0$. On se demande s'il existe $F \in \mathcal{E}_{n + p}^{l}$ et $y \in \mathcal{E}_{n}^{p}$ tels que $F(x , y(x)) = 0 $, $T_{\underline{0}} F =
\hat{F}$ et $T_{\underline{0}} y = \hat{y}$. Si $\hat{F}$ d\'epend effectivement de $x$, i.e. $\frac{\partial \hat{F}}{x_{i}} \neq 0$  pour un certain indice  $i$, la solution est simple. On se donne une r\'ealisation $\tilde{F}$ de $\hat{F}$ et une r\'ealisation $x \mapsto y ( x )$  de $\hat{y}$. L'application $x \mapsto \tilde{F} ( x , y ( x )) = a ( x )$  est plate \`a l'origine, au sens o\`u chaque composante $a_{i}$ de $a$ est plate. On pose alors $F ( x , y ) = \tilde{F} (x , y ) - a ( x)$; visiblement $T_{\underline{0}} F = \hat{F}$ et $F (x , y (x ) ) = 0 $. Lorsque $\hat{F} = \hat{F} ( y )$ ne d\'epend pas de $x$ et que l'on dispose d'une solution formelle $\hat{y} ( x )$: $\hat{F} ( \hat{y} ( x ) ) = 0$, le probl\`eme semble d\'elicat. En petite dimension ( $n = p = l = 1$) on obtient le r\'esultat suivant:

\begin{theorem} Soient $\hat{F} \in \hat{\mathcal{E}}_{2}$ et $( \hat{y}_{1} , \hat{y}_{2}) \in \hat{\mathcal{E}}_{1}^{2}$ une solution param\'etrique de $\hat{F} = 0$, i.e. $\hat{F}( \hat{y}_{1} , \hat{y}_{2}) = 0$. Il existe $F \in \mathcal{E}_{2}$ et $(y_{1}, y_{2}) \in \mathcal{E}_{1}^{2}$ satisfaisant $F(y_{1}, y_{2}) = 0$, $T_{\underline{0}} F = \hat{F}$ et $T_{\underline{0}}( y_{1} , y_{2} ) = ( \hat{y}_{1} , \hat{y}_{2})$.
\end{theorem}

\begin{proof}
On suppose que $\hat{F}$ et $(\hat{y}_{1}, \hat{y}_{2})$ sont non constants. Le Th\'eor\`eme 4.4 de D. Cerveau et J.F. Mattei \cite{Cerveau2} s'adapte facilement: il existe un diff\'eomorphisme formel $\hat{\Phi} \in \widehat{\mathrm{Diff}}(\mathbb{R}^{2}_{0})$ tel que $\hat{F} \circ \hat{\Phi} = P$ soit un polyn\^ome. Si $(\hat{y}_{1} , \hat{y}_{2})$ est une solution de $\hat{F} = 0$, alors $(\hat{Y}_{1}, \hat{Y}_{2}) = \hat{\Phi}^{- 1} (\hat{y}_{1} , \hat{y}_{2}) \in \hat{\mathcal{E}}_{1}^{2}$ est solution de $P (\hat{Y}_{1} , \hat{Y}_{2}) = 0$. Le th\'eor\`eme de Tougeron assure qu'il existe une solution $( Y_{1} , Y_{2} )  \in \mathcal{E}_{1}^{2}$  \`a $P ( Y_{1} , Y_{2} ) = 0$ avec $T_{\underline{0}} ( Y_{1} , Y_{2} ) = ( \hat{Y}_{1} , \hat{Y}_{2} )$. Soit $\Phi$ une r\'ealisation $\mathcal{C}^{\infty}$ de $\hat{\Phi}$. On v\'erifie que $F = P \circ \Phi^{- 1}$ et $( y_{1} , y_{2} ) = \Phi ( Y_{1} , Y_{2} )$ satisfont l'\'enonc\'e du th\'eor\`eme. 
\end{proof}

\begin{conjecture}
Le probl\`eme pr\'ec\'edent a une r\'eponse positive en toute g\'en\'eralit\'e.
\end{conjecture}

Un r\'esultat classique de Malgrange \cite{Malgrange} assure qu'un champ de vecteurs analytique $X \in \mathcal{X}_{2}$, poss\'edant une int\'egrale premi\`ere formelle $\hat{X}$ non constante, en poss\`ede une analytique. Ce r\'esultat ne persiste pas en classe $\mathcal{C}^{\infty}$. En effet soit $X = x_{2} \frac{\partial}{\partial x_{1}} - x_{1}\frac{\partial}{\partial x_{2}} - \exp{\frac{- 1}{x_{1}^{2} + x_{2}^{2}}} (x_{1} \frac{\partial}{\partial x_{1}} + x_{2}\frac{\partial}{\partial x_{2}})$. Il poss\`ede l'int\'egrale premi\`ere formelle $x_{1}^{2} + x_{2}^{2}$ (i.e. $T_{\underline{0}} X ( x_{1}^{2} + x_{2}^{2} ) = 0$) mais il ne poss\`ede aucune int\'egrale premi\`ere $\mathcal{C}^{\infty}$ non constante. La raison est que toutes les trajectoires de $X$ adh\`erent au point singulier $\underline{0}$ (spirales). Par contre on a l'\'enonc\'e suivant, toujours propre \`a la dimension deux:

\begin{theorem} Soient $\hat{X}$ et $\hat{F}$ deux \'el\'ements de $\hat{\mathcal{X}}_{2}$ et $\hat{\mathcal{E}}_{2}$ respectivement. On suppose que $\hat{F}$ est une int\'egrale premi\`ere non constante de $\hat{X}$ ( $\hat{X} .\hat{F} = 0$). Il existe $X \in \mathcal{X}_{2}$ et $F \in \mathcal{E}_{2}$ tels que $X . F = 0$, $T_{\underline{0}} X = \hat{X}$ et $T_{\underline{0}} F = \hat{F}$.
\end{theorem}

\begin{proof}
Soit $\hat{\Phi} \in \widehat{\mathrm{Diff}}_{2}$ un diff\'eomorphisme formel tel que $\hat{F} \circ \hat{\Phi} = P$ soit un polyn\^ome. Le champ $\hat{Y} = \hat{\Phi}^{- 1}_{\ast}\hat{X} = \sum \hat{Y}_{j} \frac{\partial }{\partial x_{j}}$ a $P$ pour int\'egrale premi\`ere:
$$\sum \hat{Y}_{j} \frac{\partial P}{\partial x_{j}} = 0$$
cette \'egalit\'e peut \^etre interpr\'et\'ee comme une relation lin\'eaire entre les polyn\^omes $\frac{\partial P}{\partial x_{j}}$. Le Th\'eor\`eme 29 produit un champ de vecteurs $Y = \sum Y_{j} \frac{\partial }{\partial x_{j}}$ de classe $\mathcal{C}^{\infty}$ tel que $Y. P = 0$. Soit $\Phi$ une r\'ealisation de Borel de $\hat{\Phi}$; le champ $X = \Phi_{\ast} Y$ et la fonction $F = P \circ \Phi^{- 1}$ conviennent. 
\end{proof}

Dans le m\^eme ordre d'id\'ee le probl\`eme des s\'eparatrices a \'et\'e consid\'er\'e par plusieurs auteurs notamment  Dumortier \cite{Dumortier}, Roussarie \cite{Roussarie} \cite{Kelley}, $\ldots$ etc. Soit $X$ un champ de vecteur $\mathcal{C}^{\infty}$ \`a l'origine $\underline{0}$ de $\mathbb{R}^{2}$ \`a singularit\'e alg\'ebriquement isol\'ee, $T_{\underline{0}} X = \hat{X}$ son jet de Taylor en $\underline{0}$ et $\omega = i_{X} d x_{1} \wedge d x_{2} \in \Omega^{1}_{2}$. Une s\'eparatrice formelle  de $X$ (ou de $\omega$) est une courbe param\'etr\'ee formelle $\hat{\gamma} = ( \hat{\gamma}_{1} , \hat{\gamma}_{2} ) \in \hat{\mathcal{E}}^{2}_{1}$, non constante telle que $\hat{\gamma} ( 0) = \underline{0}$ et $\hat{\gamma}^{\ast} \hat{\omega} = 0$.
Le th\'eor\`eme de r\'eduction des singularit\'es de Seidenberg \cite{Mattei} ajout\'e \`a  l'\'etude locale des singularit\'es apparaissant en fin de processus de r\'eduction permet d'\'etablir avec les notations pr\'ec\'edentes le:

\begin{theorem} Soient $X$ un champ de vecteur $\mathcal{C}^{\infty}$, \`a singularit\'e isol\'ee \`a l'origine de $\mathbb{R}^{2}$ et $\hat{\gamma}$ une s\'eparatrice formelle non triviale de $X$. Il existe une s\'eparatrice $\mathcal{C}^{\infty}$, $\gamma \in \mathcal{E}_{1}^{2}$ de $X$ r\'ealisant $\hat{\gamma}$: $T_{\underline{0}} \gamma = \hat{\gamma}$ et $\gamma^{\ast} \omega = 0$.
\end{theorem}

Faisons quelques commentaires sur l'\'enonc\'e ci-dessus. Pour une telle courbe $\gamma$, il existe $F \in \mathcal{E}_{2}$, irr\'eductible \`a singularit\'e alg\'ebriquement isol\'ee telle que $F^{- 1} ( 0 )$ soit \'egale \`a l'image de $\gamma$. En effet il existe $\hat{F}  \in \hat{\mathcal{E}}_{2}$, \`a singularit\'e alg\'ebriquement isol\'ee telle que $\hat{F} \circ \hat{\gamma} = 0$. La param\'etrisation est \`a reparam\'etrisation pr\`es par un \'el\'ement $\hat{\tau} \in \hat{\mathcal{E}}_{1}$, $\hat{\gamma} = \hat{\delta} \circ \hat{\tau}$ l'unique courbe formelle v\'erifiant cela. D'un autre c\^ot\'e $\hat{\gamma}$ \'etant fix\'e l'id\'eal $\langle \hat{F} \rangle$ engendr\'e par $\hat{F}$ est intrins\`eque. Soit $F^{'}$ une r\'ealisation $\mathcal{C}^{\infty}$ de $\hat{F}$. Le morphisme $\pi$ de r\'esolution des singularit\'es de $\hat{\gamma}$ (qui est \'egalement celui de $\hat{F}$) fait que $( F^{'} \circ \pi = 0 )$ est une courbe \`a croisements normaux. Soient $\gamma^{'}$ d\'efini par $\pi \circ \gamma^{'} = \gamma$. On peut modifier localement $F^{'} \circ \pi$ par composition \`a droite par un germe de diff\'eomorphisme $\varphi$, plat \`a l'identit\'e le long du diviseur $\pi^{- 1} ( \underline{0} )$ de sorte que $F^{'} \circ \pi \circ \varphi$ s'annule pr\'ecisement sur $\gamma^{'}$ ($F^{'} \circ \pi \circ \varphi \circ \gamma^{'} \equiv 0$). La fonction $F$ d\'efinie par $F \circ \pi = F^{'} \circ \pi \circ \varphi$ est $\mathcal{C}^{\infty}$, en dehors de l'origine, et $( F - F^{'} ) \circ \pi$ est plat le long de $\pi^{- 1} ( \underline{0} )$. Il en r\'esulte que $F$ s'\'etend de mani\`ere $\mathcal{C}^{\infty}$ en $\underline{0} \in \mathbb{R}^{2}$. \\

\begin{conjecture} 
Soient $X$ un germe de champ de vecteurs $\mathcal{C}^{\infty}$ \`a l'origine de $\mathbb{R}^{n}$, \`a singularit\'e alg\'ebriquement isol\'ee et poss\'edant une courbe formelle invariante $\hat{\gamma} \in \hat{\mathcal{E}}_{1}^{n}$. Alors il existe $\gamma \in \mathcal{E}_{1}^{n}$ tel que $T_{\underline{0}} \gamma = \hat{\gamma}$ et $\gamma$ est une courbe invariante de $X$.
\end{conjecture}

Le contre-exemple \`a l'existence d'int\'egrale premi\`ere $\mathcal{C}^{\infty}$  en pr\'esence d'int\'egrale premi\`ere formelle est d\^u au ph\'enom\`ene de spiralement. La pr\'esence de s\'eparatrice formelle (et donc $\mathcal{C}^{\infty}$) interdit ce spiralement. R. Roussarie \cite{Roussarie} a d\'emontr\'e que si le champ $\mathcal{C}^{\infty}$ $X$, \`a singularit\'e isol\'ee \`a l'origine de $\mathbb{R}^{2}$, poss\`ede une s\'eparatrice formelle et une int\'egrale premi\`ere formelle $\hat{F}$ (ce qui revient \`a demander que les z\'eros formels de $\hat{F}$ soient non r\'eduits \`a $\{ \underline{0} \}$), alors $X$ poss\`ede une int\'egrale premi\`ere $F$, de classe $\mathcal{C}^{\infty}$, non plate. On peut pr\'eciser cet \'enonc\'e pour obtenir la condition $T_{\underline{0}} F = \hat{F}$ (en travaillant sur les int\'egrale premi\`eres minimales formelles pour le voir). \\

Soit $S$ une s\'eparatrice du champ $X$, \`a singularit\'e isol\'ee \`a l'origine de $\mathbb{R}^{2}$. Si $S$ est d\'efinie par $( F = 0 )$, o\`u $F \in \mathcal{E}_{2}$ est \`a singularit\'e isol\'ee, alors $X. F$ est divisible par $F$, i.e. $X. F = g. F$ pour un certain $g \in \mathcal{E}_{2}$. De fa\c{c}on duale, si $\omega$ est une 1-forme duale de $X$, alors $\omega \wedge d F = F. \eta$ pour une certaine 2-forme $\eta$. Si l'on consid\`ere le champ $X = x_{2} \frac{\partial}{\partial x_{1}} - x_{1}\frac{\partial}{\partial x_{2}} - \exp{\frac{- 1}{x_{1}^{2} + x_{2}^{2}}} (x_{1} \frac{\partial}{\partial x_{1}} + x_{2}\frac{\partial}{\partial x_{2}})$ on constate que $X. ( x_{1}^{2} + x_{2}^{2} ) = -2 ( x_{1}^{2} + x_{2}^{2} )\exp{\frac{- 1}{x_{1}^{2} + x_{2}^{2}}}$.\\

On dira que $F \in \mathcal{E}_{2}$, \`a singularit\'e isol\'ee, est une s\'eparatrice elliptique du champ $X$ si $X. F$ est divisible par $F$ et $( F = 0 )$ se r\'eduit \`a  $\{ \underline{0} \}$: $X. F = g . F$ et $F^{- 1} ( 0 ) = \{ \underline{0} \}$. Lorsque $g$ lui m\^eme est elliptique, i.e. $g^{- 1} ( 0 ) = \{ \underline{0} \}$, alors $F$ est une fonction de Lyapounov pour $X$; dans ce cas toutes les trajectoires du champ $X$ adh\'erent \`a l'origine. On a une notion formelle analogue. Soit $\hat{X}$ un champ formel \`a singularit\'e isol\'ee et $\hat{F} \in \hat{\mathcal{E}}_{2}$ \'egalement \`a singularit\'e isol\'ee. Alors $\hat{F}$ d\'efinit une s\'eparatrice formelle elliptique de $\hat{X}$ si $\hat{X}. \hat{F}$ est divisible par $\hat{F}$ ($\hat{X} . \hat{F} = \hat{g}. \hat{F}$ o\`u $\hat{g} \in \hat{\mathcal{E}}_{2}$) et $\hat{F}^{- 1} ( 0 ) = \{ \underline{0} \}$ au sens o\`u $\hat{F} (\underline{0} ) = 0$ et pour tout chemin formel $\hat{\gamma} = ( \hat{\gamma}_{1} , \hat{\gamma}_{2} )$, $\hat{\gamma}_{i} \in \hat{\mathcal{E}}_{1}$, tel que $\hat{F} \circ \hat{\gamma} \equiv 0$; alors $\hat{\gamma}$ est le chemin "constant" $\hat{\gamma} ( t ) \equiv 0$.

\begin{proposition} Soient $\hat{X}$ un champ formel et $\hat{F}$ une s\'eparatrice formelle elliptique de $\hat{X}$. Il existe des r\'ealisations $X$ et $F$ de $\hat{X}$ et $\hat{F}$ respectivement telles que $F$ soit une s\'eparatrice $\mathcal{C}^{\infty}$ et elliptique de $X$.
\end{proposition}

\begin{proof}
Soit $\hat{\Phi} \in \widehat{\mathrm{Diff}}(\mathbb{R}^{2}_{0})$ tel que $\hat{F} \circ \hat{\Phi} = P$ soit un polyn\^ome. On se ram\`ene donc \`a $\hat{X}^{'} . P = P. \hat{g}$, o\`u $\hat{X}^{'}$ est le transform\'e de $\hat{X}$ par $\hat{\Phi}$. Le th\'eor\`eme de Brian\c{c}on et Skoda \cite{Briancon} affirme que $P^{2}$ appartient \`a son id\'eal Jacobien (ceci dans la classe analytique). Ce qui s'interpr\`ete comme suit: il existe un champ de vecteur analytique $Y$ tel que $P^{2} = Y . P$. Notons que $\hat{X} . P^{2} = 2 P^{2} \hat{g} = 2 \hat{g} . Y . P$. Par suite: $( 2 \hat{g} Y - 2 P \hat{X}^{'} ) . P = 0$ qui s'\'ecrit avec des notations \'evidentes sous forme:
$$ 2 \hat{g} ( Y_{1} \frac{\partial P}{\partial x_{1}} + Y_{2} \frac{\partial P}{\partial x_{2}} ) - 2 P  \frac{\partial P}{\partial x_{1}} . \hat{X}^{'}_{1} - 2 P  \frac{\partial P}{\partial x_{2}} . \hat{X}^{'}_{2} = 0$$
que l'on peut voir comme une \'equation lin\'eaire en les inconnues $g$, $X^{'}_{1}$ et $X^{'}_{2}$ ayant une solution formelle $\hat{g}$, $\hat{X}^{'}_{1}$ et $\hat{X}^{'}_{2}$. Suivant Malgrange il existe des r\'ealisations $\mathcal{C}^{\infty}$  $g$, $X^{'}_{1}$ et $X^{'}_{2}$ telle que $( 2 g Y - 2 P X^{'} ) . P = 0$ avec $X^{'} = X^{'}_{1} \frac{\partial}{\partial x_{1}} + X^{'}_{2} \frac{\partial}{\partial x_{2}}$. On en d\'eduit que $X^{'}. P = g P$. On choisit une r\'ealisation de Borel de $\hat{\Phi}$; le transform\'e $X$ de $X^{'}$ par $\Phi^{- 1}$ et $F = P \circ \Phi^{- 1}$ sont des r\'ealisations de $\hat{X}$ et $\hat{F}$ qui conviennent. 
\end{proof} 

\section{Probl\`emes mixtes}

\subsection{Le cas des groupes}

Etant donn\'e un sous groupe $\hat{G} \subset \widehat{\mathrm{Diff}}(\mathbb{R}^{n}_{0})$ de diff\'eomorphismes formels, on cherche un sous groupe $G \subset \mathrm{Diff}(\mathbb{R}^{n}_{0})$ de diff\'eomorphismes $\mathcal{C}^{\infty}$ tels que l'application $T_{\underline{0}} : G \rightarrow \hat{G}$ soit un isomorphisme. Nous ne savons pas r\'epondre \`a cette question en toute g\'en\'eralit\'e, mais voici quelques exemples o\`u la r\'eponse est positive.

\subsubsection{Le cas des sous groupes libres}

Consid\'erons un groupe libre engendr\'e par des \'el\'ements $\hat{f}_{1} , \ldots , \hat{f}_{m} \in \widehat{\mathrm{Diff}}(\mathbb{R}^{n}_{0})$. Si $f_{1} , \ldots , f_{m}$ sont des r\'ealisations quelconques de $\hat{f}_{1} , \ldots , \hat{f}_{m}$ respectivement ($T_{\underline{0}} f_{k} = \hat{f}_{k}$), alors le groupe $G$ engendr\'e par les $f_{k}$ est libre. En effet si $M (f_{1} , \ldots , f_{m}) = \mathrm{Id}_{\mathbb{R}^{n}}$ est un mot r\'eduit en les $f_{k}$ \'egal \`a l'identit\'e, alors $M (\hat{f}_{1} , \ldots , \hat{f}_{m} ) = \mathrm{Id}_{\mathbb{R}^{n}}$ et le mot est donc trivial. C'est donc un fait \'el\'ementaire qui fait que $T_{\underline{0}} : G \rightarrow \hat{G}$ soit un isomorphisme.

\subsubsection{Le cas des produits libres de groupes finis}

 On se donne un sous groupe $\hat{G} = \hat{G}_{1} \ast \ldots \ast \hat{G}_{m} \subset \widehat{\mathrm{Diff}}(\mathbb{R}^{n}_{0})$, produit libre d'un nombre fini de sous groupes finis $\hat{G}_{k} \subset \widehat{\mathrm{Diff}}(\mathbb{R}^{n}_{0})$. Rappelons que tout sous groupe fini de $\widehat{\mathrm{Diff}}(\mathbb{R}^{n}_{0})$ est conjugu\'e \`a sa partie lin\'eaire. Pour chacun des groupes $G_{k}$ il existe $\hat{\Phi}_{k} \in \widehat{\mathrm{Diff}}(\mathbb{R}^{n}_{0})$ tels que $\hat{G}_{k} = \hat{\Phi}_{k} \circ G^{1}_{k} \circ \hat{\Phi}_{k}^{- 1}$, o\`u $G^{1}_{k}$ est le sous groupe lin\'eaire des 1-jets de $\hat{G}_{k}$. Si  les $\Phi_{k}$ sont des r\'ealisations des $\hat{\Phi}_{k}$  on consid\`ere les sous groupes  $G_{k} = \Phi_{k} \circ G^{1}_{k} \circ \Phi_{k}^{- 1}$. Soit $G$ le groupe engendr\'e par les $G_{k}$. Par construction le morphisme $T_{\underline{0}} : G \rightarrow \hat{G}$ est surjectif et le groupe $G$ est le produit libre des groupes $G_{k}$. Par suite $T_{\underline{0}} : G \rightarrow \hat{G}$ est un isomorphisme.
Ce r\'esultat se g\'en\'eralise au cas o\`u $\hat{G}$ est le produit libre de groupes lin\'earisables $\hat{G}_{k}$; la d\'emarche est analogue \`a celle qui pr\'ec\`ede: on choisit des r\'ealisations $\Phi_{k}$ des $\hat{\Phi}_{k}$ et l'on introduit les groupes $G_{k} = \Phi_{k} \circ G^{1}_{k} \circ \Phi_{k}^{- 1}$, o\`u $G^{1}_{k}$ est le groupe lin\'eaire $J^{1} \hat{G}_{k}$. Le groupe $G$ engendr\'e par les $G_{k}$ est le produit libre $G_{1} \star \ldots \star G_{m}$ et $T_{\underline{0}} : G \rightarrow \hat{G}$ est un isomorphisme.\\

En dimension 1 la classification des sous groupes r\'esolubles de diff\'eomor-phismes formels est bien connue et fait partie du folklore \cite{Cerveau1}. En fait tout sous groupe $\hat{G} \subset \widehat{\mathrm{Diff}}(\mathbb{R}^{1}_{0})$, r\'esoluble, est formellement conjugu\'e \`a un sous groupes de diff\'eomorphismes analytique. On en d\'eduit le:

\begin{theorem} Soit $\hat{G} \subset \widehat{\mathrm{Diff}}(\mathbb{R}^{1}_{0})$ un sous groupe r\'esoluble; il existe un sous groupe $G \subset \mathrm{Diff}(\mathbb{R}^{1}_{0})$ tel que $T_{\underline{0}} G = \hat{G}$ et $T_{\underline{0}} : G \rightarrow \hat{G}$ soit un isomorphisme.
\end{theorem}

Voici encore un cas sp\'ecial qui s'applique en particulier aux repr\'esentations des groupes de surfaces dans $\widehat{\mathrm{Diff}}(\mathbb{R}^{n}_{0})$. On se donne des diff\'eomorphismes formels $\hat{a} , \hat{b} , \hat{c}_{1} , \ldots , \hat{c}_{p} \in \widehat{\mathrm{Diff}}(\mathbb{R}^{n}_{0})$ satisfaisant la relation:

$$ \hat{a}  \hat{b}  \hat{a}^{-1} = \mathcal{R} ( \hat{b} , \hat{c}_{1} , \ldots , \hat{c}_{p})$$
o\`u $\mathcal{R}$ est un mot en les $\hat{b} , \hat{c}_{1} , \ldots \hat{c}_{p}$. Une telle relation indique que $\hat{b}$ est conjugu\'e \`a $\mathcal{R} ( \hat{b} , \hat{c}_{1} , \ldots , \hat{c}_{p})$ par le diff\'eomorphisme formel $\hat{a}$. Soit $b_{1} = J^{1} \hat{b}$ la partie lin\'eaire de $\hat{b}$; nous supposons  $b_{1}$ semi-simple et hyperbolique. Ceci signifie que si $\lambda_{1} , \ldots , \lambda_{n}$ sont les valeurs propres (complexes) de $b_{1}$ alors aucun des $\lambda_{j}$ n'est imaginaire pur.

\begin{proposition} Sous les hypoth\`eses pr\'ec\'edentes il existe $a , b , c_{1} , \ldots , c_{p} \in \mathrm{Diff}(\mathbb{R}^{n}_{0})$ r\'ealisations $\mathcal{C}^{\infty}$ de
$\hat{a} , \hat{b} , \hat{c}_{1} , \ldots , \hat{c}_{p}$ et satisfaisant \`a la relation:

$$ a  b  a^{-1} = \mathcal{R} ( b , c_{1} , \ldots , c_{p}) .$$
\end{proposition}

\begin{proof}
Soient $ b , c_{1} , \ldots , c_{p}$ des r\'ealisations $\mathcal{C}^{\infty}$ des diff\'eomorphismes formels $\hat{b} , \hat{c}_{1} , \ldots , \hat{c}_{p}$. Les diff\'eomorphismes $b$ et $\mathcal{R} ( b , c_{1} , \ldots , c_{p})$ sont formellement conjugu\'es. Comme $b$ est hyperbolique ils sont $\mathcal{C}^{\infty}$ conjugu\'es par un diff\'eomorphismes $a$ ayant $\hat{a}$ comme jet infini; c'est le Th\'eor\`eme de lin\'earisation de Sternberg \cite{Shlomo1}, \cite{Shlomo2}. 
\end{proof}

En dimension 1, il y a un r\'esultat de normalisation d\^u \`a F. Takens \cite{Takens}:

\begin{theorem} Soit $f \in \mathrm{Diff}(\mathbb{R}^{1}_{0})$ un germe de diff\'eomorphisme $\mathcal{C}^{\infty}$. On suppose que $f - \mathrm{Id}_{\mathbb{R}}$ n'est pas plat:
$f - \mathrm{Id}_{\mathbb{R}} \notin \mathcal{M}_{1}^{\infty}$. Alors si $g \in \mathrm{Diff}(\mathbb{R}^{1}_{0})$ est formellement conjugu\'e \`a $f$ par $\hat{h}$:
$$\hat{g} = \hat{h} \circ \hat{f} \circ \hat{h}^{- 1}$$
il existe $h \in \mathrm{Diff}(\mathbb{R}^{1}_{0})$ tel que $T_{0} h = \hat{h}$ et $g = h \circ f \circ h^{- 1}$.
\end{theorem}

Ceci permet d'avoir une variante de l'\'enonc\'e pr\'ec\'edent.

\begin{proposition} Soient $\hat{a} , \hat{b} , \hat{c}_{1} , \ldots , \hat{c}_{p}$ des \'el\'ements non triviaux de $\widehat{\mathrm{Diff}}(\mathbb{R}^{1}_{0})$ satisfaisant une relation du type:
 $$ \hat{a}  \hat{b}  \hat{a}^{-1} = \mathcal{R} ( \hat{b} , \hat{c}_{1} , \ldots , \hat{c}_{p}) .$$
 Il existe des r\'ealisations $\mathcal{C}^{\infty}$ des $\hat{a}$, $\hat{b}$, $\hat{c}_{j}$ satisfaisant la m\^eme relation.
\end{proposition}

\begin{remark} Cet \'enonc\'e s'applique en particulier aux repr\'esentations $\hat{\rho} : G_{n} \rightarrow \widehat{\mathrm{Diff}}(\mathbb{R}^{1}_{0})$ du groupe fondamental
$G_{n}$ de la surface de genre $n$. Un tel $\hat{\rho}$ poss\`ede une r\'ealisation $\mathcal{C}^{\infty}$ $\rho : G_{n} \rightarrow \mathrm{Diff}(\mathbb{R}^{1}_{0})$.
\end{remark}

Les exemples pr\'ec\'edents conduisent \`a la: \\

\begin{conjecture}
Soit $\hat{G} \subset \widehat{\mathrm{Diff}}(\mathbb{R}^{n}_{0})$ un sous groupe de diff\'eomorphismes formels. Il existe une r\'ealisation $\mathcal{C}^{\infty}$, $G \subset \mathrm{Diff}(\mathbb{R}^{n}_{0})$ de $\hat{G}$ telle que $T_{\underline{0}} : G \rightarrow \hat{G}$ soit un isomorphisme.
\end{conjecture}

\subsection{Feuilletages de codimension 1}

Consid\'erons  $\hat{\omega} \in \hat{\Omega}^{1}_{n}$, $n \geq 3$, une 1-forme formelle int\'egrable ($\hat{\omega} \wedge d \hat{\omega} = 0$) \`a singularit\'e alg\'ebriquement isol\'ee; i.e. si $\hat{\omega} = \sum_{i = 1}^{n} \hat{a}_{i} d x_{i}$, $\hat{a}_{i} \in \hat{\mathcal{E}}_{n}$, alors le quotient $\hat{\mathcal{E}}_{n} / (\hat{a}_{1} , \ldots , \hat{a}_{n})$ est de dimension finie. Le Th\'eor\`eme de Malgrange formel (d\^u en fait \`a J. Martinet) assure l'existence  d'une int\'egrale premi\`ere formelle $\hat{f}$:  $\hat{\omega} = \hat{g} d \hat{f}$ avec $\hat{f}$ et $\hat{g}$ appartenant \`a $\hat{\mathcal{E}}_{n}$. En fait il suffit qu'il existe une immersion $i : \mathbb{R}^{3}_{, 0} \rightarrow \mathbb{R}^{n}_{, 0}$ d'un trois plan tel que $i^{\ast} \hat{\omega}$ soit \`a singularit\'e isol\'ee pour arriver \`a la m\^eme conclusion. Soient $f$ et $g$  des r\'ealisations $\mathcal{C}^{\infty}$ de $\hat{f}$ et $\hat{g}$ respectivement. La 1-forme $\omega = g d f$ est int\'egrable et satisfait $T_{\underline{0}} \omega = \hat{\omega}$.
Dans la description des singularit\'es analytiques des feuilletages de codimension 1 appara\^{i}t  naturellement les feuilletages obtenus par image r\'eciproque de feuilletages en petite dimension par un morphisme. Il en est ainsi pour les "ph\'enom\`enes de Kupka" \cite{Cerveau2}. Si $\hat{\omega}_{0} \in \Omega_{2}^{1}$ d\'efinit un feuilletage formel de $\mathbb{R}^{2}_{, 0}$ et si $\hat{F} = (\hat{F}_{1} , \hat{F}_{2}) \in \hat{\mathcal{E}}_{n}^{2}$ est un morphisme formel de $\mathbb{R}^{n}_{, 0}$ dans $\mathbb{R}^{2}_{, 0}$, alors $\hat{F}^{\ast}\hat{\omega}_{0}$ d\'efinit un feuilletage formel de $\mathbb{R}^{n}_{, 0}$. Si $F$ et $\omega_{0}$ sont des r\'ealisations $\mathcal{C}^{\infty}$ de $\hat{F}$ et $\hat{\omega}_{0}$, alors $\omega = F^{\ast} \omega_{0}$ est une r\'ealisation $\mathcal{C}^{\infty}$ int\'egrable de $\hat{\omega}$.

D'autres feuilletages apparaissant dans la classification abondent dans le sens pr\'ec\'edent. Par exemple soit $\hat{\omega} \in \hat{\Omega}_{n}^{1}$ de type "logarithmique" $\hat{\omega} = \hat{f}_{1} \ldots \hat{f} \sum_{i = 1}^{p} \lambda_{i} \frac{d \hat{f}_{i}}{\hat{f}_{i}}$ o\`u $\lambda_{i} \in \mathbb{R}$ et $\hat{f}_{i} \in \hat{\mathcal{E}}_{n}$. Si  les $f_{i} \in \mathcal{E}_{n}$ sont des r\'ealisations $\mathcal{C}^{\infty}$ des $\hat{f}_{i}$, alors $\omega = f_{1} \ldots f \sum_{i = 1}^{p} \lambda_{i} \frac{d f_{i}}{f_{i}}$ est une r\'ealisation $\mathcal{C}^{\infty}$ int\'egrable de $\hat{\omega}$. Ici encore nous proposons la
conjecture:

\begin{conjecture}
Soit $\hat{\omega} \in \hat{\Omega}_{n}^{1}$ une 1-forme formelle int\'egrable. Il existe $\omega \in \Omega_{n}^{1}$ int\'egrable telle que $T_{\underline{0}} \omega = \hat{\omega}$.
\end{conjecture}

D. Cerveau \\
Universit\'e de Rennes1 \\
IRMAR, CNRS UMR 6625 \\
Campus de Beaulieu, B\^at. 22-23 \\
F-35042 Rennes Cedex, France\\
dominique.cerveau@univ-rennes1.fr \\
\\

D. Garba Belko \\
Universit\'e Abdou Moumouni \\
Facult\'e des Sciences\\
D\'epartement de Math\'ematiques\\
B.P. 10662 Niamey, Niger \\
garbabelkodjibrilla@yahoo.fr
\end{document}